\documentclass[5p,preprint]{elsarticle}

\usepackage[active]{srcltx}
\usepackage{amsmath}
\usepackage{amssymb}
\usepackage{amsfonts}
\usepackage{amsthm}
\usepackage{tikz}
\usepackage{graphicx,pst-all}
\usepackage{subfig}
\usepackage{enumerate}
\usepackage[shortlabels]{enumitem}
\usepackage{caption}
\usepackage{comment}

\newcommand{\cC}{\mathcal{C}}

\newcommand{\cW}{\mathcal{W}}

\newcommand{\cE}{\mathcal{E}}

\newcommand{\cL}{\mathcal{L}}

\usetikzlibrary{arrows}
\usetikzlibrary{fit}\usetikzlibrary{calc}
  \pgfdeclarelayer{background}
  \pgfsetlayers{background,main}

\theoremstyle{theorem}
\newtheorem{Prop}{Proposition}[section]
\newtheorem{Lem}[Prop]{Lemma}
\newtheorem{Thm}[Prop]{Theorem}

\theoremstyle{definition}

\newtheorem{Rem}[Prop]{Remark}

\newtheorem{Ex}[Prop]{Example}

\newcommand{\eps}{\varepsilon}

\newcommand{\N}{{\mathbb{N}}}

\newcommand{\R}{{\mathbb{R}}}
\newcommand{\C}{{\mathbb{C}}}

\newcommand{\Gl}{\mathbf{Gl}}

\newcommand{\la}{\lambda}

\DeclareMathOperator{\rk}{rk}

\DeclareMathOperator{\im}{im}

\newcommand{\setdef}[2]{\left\{\ #1\ \left|\ \vphantom{#1} #2\ \right.\right\}}
\newcommand{\ddt}{\tfrac{\text{\normalfont d}}{\text{\normalfont d}t}}

\newcommand{\ds}[1]{{\rm \, d} #1 \,}

{\left(\begin{smallmatrix}}
{\end{smallmatrix}\right)}

\newenvironment{smallbmatrix}
{\left[\begin{smallmatrix}}
{\end{smallmatrix}\right]}

\newlength{\innersep}
\newlength{\maxlength}
\newlength{\dummylength}

\newcommand{\JordanBlock}[3]{
\setlength{\arraycolsep}{0pt}
\renewcommand{\arraystretch}{0}
\settowidth{\maxlength}{$#1$}
\settoheight{\dummylength}{$#1$}
\ifdim\dummylength>\maxlength
  \setlength{\maxlength}{\dummylength}
\fi
\settowidth{\dummylength}{$#2$}
\ifdim\dummylength>\maxlength
  \setlength{\maxlength}{\dummylength}
\fi
\settoheight{\dummylength}{$#2$}
\ifdim\dummylength>\maxlength
  \setlength{\maxlength}{\dummylength}
\fi
\setlength{\innersep}{0.1\maxlength}
\addtolength{\maxlength}{\innersep}
\addtolength{\maxlength}{\innersep}
\newcommand{\invisiblebox}{\phantom{\rule{\maxlength}{\maxlength}}}
\begin{array}{cccc}
  \tikz[remember picture] \node[outer sep=0,inner sep=\innersep] (a11) {$#1$}; &{\tikz[remember picture] \node[outer sep=0,inner sep=\innersep] (a21) {$#2$};}& & \invisiblebox\\
  &&\phantom{\rule{#3}{#3}} &\\
   \invisiblebox &\invisiblebox& &{\tikz[remember picture] \node[outer sep=0,inner sep=\innersep] (a43) {$#2$};} \\
   \invisiblebox &\invisiblebox& &
   {\tikz[remember picture] \node[outer sep=0,inner sep=\innersep] (a44) {$#1$};}
\end{array}
\tikz[remember picture, overlay] \draw (a11) edge[very thick] (a44);
\tikz[remember picture, overlay] \draw (a21) edge[very thick] (a43);
}

\begin{document}

\begin{frontmatter}

\title{Tracking with prescribed performance for linear non-minimum phase systems\tnoteref{thanks}}
\tnotetext[thanks]{This work was supported by the German Research Foundation (Deutsche Forschungsgemeinschaft) via the grant BE 6263/1-1.}

\author[Paderborn]{Thomas Berger}\ead{thomas.berger@math.upb.de}

\address[Paderborn]{Institut f\"ur Mathematik, Universit\"at Paderborn, Warburger Str.~100, 33098~Paderborn, Germany\vspace*{-0.8cm}}


\begin{abstract}
We consider tracking control for uncertain linear systems with known relative degree which are possibly non-minimum phase, i.e., their zero dynamics may have an unstable part. For a given sufficiently smooth reference signal we design a low-complexity controller which achieves that the tracking error evolves within a prescribed performance funnel. We present a novel approach where a new output is constructed, with respect to which the system has a higher relative degree, but the unstable part of the zero dynamics is eliminated. Using recent results in funnel control, we then design a controller with respect to this new output, which also incorporates a new reference signal. We prove that the original output stays within a prescribed performance funnel around the original reference trajectory and all signals in the closed-loop system are bounded. The results are illustrated by some simulations.
\end{abstract}

\end{frontmatter}


%
\section{Introduction}\label{Sec:Intr}
%

We study output tracking for uncertain linear non-minimum phase systems with arbitrary relative degree by funnel control. The concept of funnel control was originally developed in~\cite{IlchRyan02b}, see also the survey~\cite{IlchRyan08} and the references therein. The funnel controller is an adaptive controller of high-gain type and proved to be the appropriate tool for tracking problems in various applications, such as temperature control of chemical reactor models~\cite{IlchTren04}, control of industrial servo-systems~\cite{Hack17} and underactuated multibody systems~\cite{BergOtto19}, speed control of wind turbine systems~\cite{Hack14,Hack15b}, voltage and current control of electrical circuits~\cite{BergReis14a}, control of peak inspiratory pressure~\cite{PompWeye15} and adaptive cruise control~\cite{BergRaue18}.

The above mentioned applications have the advantage that their underlying dynamics are minimum-phase, i.e., their internal dynamics (zero dynamics in the linear case) are bounded-input, bounded-output stable. The internal dynamics and the minimum phase property are extensively studied in the literature, see e.g.~\cite{ByrnWill84, IlchWirt13, Mare84, Mors83}. A main obstacle for feedback controllers are systems which are not minimum phase, i.e., their internal dynamics have an unstable part. Such unstable parts of the internal dynamics may impose fundamental limitations on the transient tracking performance as shown in~\cite{QiuDavi93b}. These limitations were already highlighted in the seminal work by Byrnes and Isidori~\cite{ByrnIsid90b}, where they prove that the regulator problem is solvable provided that the internal dynamics of the system have a hyperbolic equilibrium. The solution is constructed from the solution of a set of partial differential-algebraic equations, which however may be very difficult to solve, if not impossible. Extending the approach from~\cite{ByrnIsid90b}, in~\cite{GopaHedr93} so called \emph{ideal internal dynamics} are used and made attractive by a suitable redefinition of the output which does not change the relative degree. Using a sliding control law, it is achieved that the new output tracks a suitably modified reference signal and in the end, the original output asymptotically tracks the original reference trajectory. However, the ideal internal dynamics require a \emph{trackability assumption}, i.e., the existence of a bounded solution of the internal dynamics when the reference signal is inserted for the output. In~\cite{ShkoShte02} the approach from~\cite{GopaHedr93} is extended by using the so called \emph{system center method}, and~\cite{SuLin11} develop further improvements. These methods aim at asymptotically obtaining the ideal internal dynamics, however sufficient conditions for their feasibility are not available.

In a different approach,~\cite{ChenPade96, DevaChen96} aim to resolve the problem imposed by unstable internal dynamics using the concept of \emph{stable system inversion}. In contrast to~\cite{ByrnIsid90b}, an open-loop (feedfoward) control input is calculated here for all times, based on the given reference trajectory. A drawback of this approach is that in the case of non-minimum phase systems, a reverse-time integration is used and hence the computed control input must start in advance to achieve the desired tracking performance. Therefore, the open-loop control input is non-causal in this case. Extensions of this approach are discussed in~\cite{Deva99, DevaPade98, HuntMeye97b, TaylLi02} for instance.

Noteworthy is also the approach presented by Isidori in~\cite{Isid00}, where stabilization of non-minimum phase systems by dynamic compensators is considered. The crucial assumption imposed in the aforementioned work is that an auxiliary system resulting from the interconnection with the compensator is itself stabilizable by dynamic output feedback. Later,~\cite{NazrKhal11} pointed out that this is equivalent to using a compensator which provides a new output with respect to which the interconnection has relative degree one. The assumption then is that the internal dynamics of the interconnection are stable. Extensions of this approach to regulator problems have been studied in~\cite{IsidMarc03, MarcIsid04, NazrKhal09}. It is an advantage of this approach that by using high-gain observers the control objective can be achieved by output feedback only. However, prescribed performance of the original tracking error cannot be achieved, not even if a funnel controller would be used in this framework, since transient bounds for the new output given by the compensator do not lead to transient bounds for the original tracking error.

Last but not least, we like to mention the approach presented in~\cite{HausSast92}, where tracking for \emph{slightly non-minimum phase} systems is considered.

In the present paper, we introduce a novel approach to treat output tracking with prescribed performance of the tracking error for uncertain linear non-minimum phase systems with arbitrary relative degree. Similar to~\cite{GopaHedr93} we define a new output for the system. However, our aim is not to keep the relative degree as it is and stabilize the internal dynamics, but to completely remove the unstable part of the internal dynamics by increasing the relative degree. The new output is a part of the former internal dynamics, and a suitable redefinition of the reference trajectory is necessary as well. To this end, we insert the original reference signal into the part of the internal dynamics which has been eliminated by the output redefinition. If the internal dynamics have a hyperbolic equilibrium, then it is possible to suitably adjust the initial value so that the solution, which provides the new reference signal, is bounded; this is different from the trackability assumption in~\cite{GopaHedr93}. Under a mild assumption, we may also allow for a non-hyperbolic equilibrium. We may then apply the funnel controller for systems with arbitrary relative degree developed in~\cite{BergHoan18} to the system with new output and new reference signal. We show that by a suitable choice of the design parameters it can be achieved that the original output stays within a prescribed performance funnel around the original reference trajectory. As far as the author is aware, another result on tracking with prescribe performance for non-minimum phase systems is not available in the literature.

We stress that a main feature of funnel control is that it is model-free (only structural assumptions on the system class are required, such as the minimum phase property) and hence inherently robust. Moreover, it was recently shown that even for higher relative degree systems funnel control is feasible using output error feedback only, and no derivatives of the output are required, see~\cite{BergReis18a, BergReis18b}. These features are partially lost when dealing with non-minimum phase systems, where some knowledge of the system parameters and measurement of additional state variables is required. The additional knowledge is used to construct the new output and reference signal to which the funnel controller is applied. Robustness with respect to a large class of uncertainties is still retained.

Throughout this article, we use the following notation: We write $\R_{\ge 0}=[0,\infty)$ and $\C_-$, $(\C_+)$ denotes the set of complex numbers with negative (positive) real part. ${\Gl}_n(\R)$ denotes the group of invertible matrices in $\R^{n\times n}$ and $\sigma(A)$ the spectrum of $A\in\R^{n\times n}$. By $\mathcal{L}^\infty(I\!\to\!\R^n)$ we denote the set of essentially bounded functions $f:I\!\to\!\R^n$ with norm $\|f\|_\infty = {\rm ess\ sup}_{t\in I} \|f(t)\|$. The set $\mathcal{W}^{k,\infty}(I\!\to\!\R^n)$ contains all $k$-times weakly differentiable functions $f:I\!\to\!\R^n$ such that $f,\ldots, f^{(k)}\in \mathcal{L}^\infty(I\!\to\!\R^n)$. By $\mathcal{C}^k(I\!\to\!\R^n)$ we denote the set of $k$-times continuously differentiable functions, where $k\in\N_0\cup\{\infty\}$.

\subsection{System class}\label{Ssec:SysClass}

We consider uncertain linear systems given by
\begin{equation}\label{eq:ABC}
\begin{aligned}
    \dot x(t) &= Ax(t) + Bu(t) + d(t),\qquad x(0)=x^0\in\R^n,\\
    y(t) &= Cx(t),
\end{aligned}
\end{equation}
where $A\in\R^{n\times n}$ and $B, C^\top \in\R^{n\times m}$, with the same number of inputs $u:\R_{\ge 0}\to\R^m$ and outputs $y:\R_{\ge 0}\to\R^m$, and $d:\R_{\ge 0}\to\R^n$ accounts for possible disturbances and uncertainties. We assume that~\eqref{eq:ABC} has strict relative degree $r\in\N$, that is
\begin{equation}\label{eq:rel-deg}
\begin{aligned}
   & CA^kB = 0,\ k=0,\ldots,r-2, \ CA^{r-1}B \in \Gl_n(\R),\\
   & CA^kd(\cdot) = 0,\ k=0,\ldots,r-2,
\end{aligned}
\end{equation}
cf.~\cite{Isid95}, and that $d\in\cL^{\infty}(\R_{\ge 0}\to\R^n)$. While adaptive control of minimum phase linear systems is well-studied, see e.g.\ the classical works~\cite{ByrnWill84, KhalSabe87, Mare84, Mors83}, we stress that we do not assume that~\eqref{eq:ABC} is minimum phase or, equivalently, its zero dynamics are asymptotically stable, cf.~\cite{IlchWirt13}. The latter would mean that $\rk \begin{smallbmatrix} A - \lambda I_n & B \\ C & 0\end{smallbmatrix} = n+m$ for all $\lambda\in\C_-$, see e.g.~\cite{IlchRyan07, Isid95}. As an important tool for the forthcoming controller design we recall the Byrnes-Isidori form for linear systems~\eqref{eq:ABC}. By a straightforward extension of~\cite[Lem.~3.5]{IlchRyan07} (see also~\cite{Isid95}) we have that, if~\eqref{eq:rel-deg} is satisfied, then there exists a state-space transformation $U\in\Gl_n(\R)$ such that $U x(t) = \big( y(t)^\top, \dot y(t)^\top, \ldots, y^{(r-1)}(t)^\top, \eta(t)^\top\big)^\top$, where $\eta:\R_{\ge 0}\to\R^{n-rm}$, transforms~\eqref{eq:ABC} into
\begin{equation}\label{eq:BIF}
\begin{aligned}
    y^{(r)}(t) &= \sum_{i=1}^r R_i y^{(i-1)}(t) + S \eta(t) + \Gamma u(t) + d_r(t),\\
    \dot \eta(t) &= P y(t) + Q\eta(t) +d_\eta(t),
\end{aligned}
\end{equation}
where $R_i\in\R^{m\times m}$ for $i=1,\ldots,r$, $S, P^\top\in\R^{m\times (n-rm)}$, $Q\in\R^{(n-rm)\times (n-rm)}$, $\Gamma := CA^{r-1} B$ and $(d_r^\top, d_\eta^\top)^\top = U d$. Furthermore,~\eqref{eq:ABC} is minimum phase if, and only if, $\sigma(Q)\subseteq\C_-$. The second equation in~\eqref{eq:BIF} represents the internal dynamics of the linear system~\eqref{eq:ABC}; if $y=0$, then these dynamics are called zero dynamics.


\subsection{Control objective}\label{Ssec:ContrObj}

To treat the non-minimum phase property of system~\eqref{eq:ABC} the system parameters $A, B, C$ need to be known, at least partially, and additional components of the state~$x$ need to be available to the controller; the required information is made precise in Sections~\ref{Sec:TrackAss} and~\ref{Sec:ContrStruc}. For the time being, assume that the measurement of a partial state $\hat x(t) = H x(t)$ is available, where~$H$ will be specified by the presented controller design. We stress that the measurement of the full state~$x(\cdot)$ or knowledge of the full initial value~$x^0$ and the disturbance~$d(\cdot)$ are, in general, not required. Therefore, the objective is to design a dynamic partial state feedback of the form
\begin{equation}\label{eq:objcontr}
\begin{aligned}
\dot{z}(t)\,&=F\big(t,z(t),\hat x(t),y_{\rm ref}(t)\big),\quad z(0) = z^0,\\
u(t)\,&=G\big(t,z(t),\hat x(t),y_{\rm ref}(t)\big),
\end{aligned}
\end{equation}
where $y_{\rm ref}:\R_{\ge 0}\to\R^m$ is a~sufficiently smooth reference signal, such that in the closed-loop system the tracking error $e(t)=y(t)-y_{\rm ref}(t)$ evolves within a prescribed performance funnel
\begin{equation}
\mathcal{F}_{\varphi} := \setdef{(t,e)\in\R_{\ge 0} \times\R^m}{\varphi(t) \|e\| < 1},\label{eq:perf_funnel}
\end{equation}
which is determined by a function~$\varphi$ belonging to
\[
\Phi_r \!:=\!
\left\{
\varphi\in  \cC^r(\R_{\ge 0}\to\R)
\left|\!\!\!
\begin{array}{l}
\text{ $\varphi, \dot \varphi,\ldots,\varphi^{(r)}$ are bounded,}\\
\text{ $\varphi (\tau)>0$ for all $\tau\ge 0$,}\\
 \text{ and }  \liminf_{\tau\rightarrow \infty} \varphi(\tau) > 0
\end{array}
\right.\!\!\!
\right\}.
\]
Furthermore, all signals~$x, u, z$ should remain bounded, even though~\eqref{eq:ABC} is non-minimum phase.

The funnel boundary is given by the reciprocal of $\varphi$ as depicted in Fig.~\ref{Fig:funnel}. In contrast to most other works on funnel control, cf.\ e.g.~\cite{BergHoan18,IlchRyan02b}, we do not allow for the case $\varphi(0)=0$ , which would mean that there is no restriction on the initial value since $\varphi(0) \|e(0)\| < 1$ and the funnel boundary $1/\varphi$ has a pole at $t=0$. For technical reasons, we require that the funnel boundary is ``finite'' at $t=0$.

\begin{figure}[h]
\vspace*{-2mm}
\captionsetup[subfloat]{labelformat=empty}
\hspace*{1cm}
\begin{tikzpicture}[scale=0.45]
\tikzset{>=latex}
  \filldraw[color=gray!15] plot[smooth] coordinates {(0.15,4.7)(0.7,2.9)(4,0.4)(6,1.5)(9.5,0.4)(10,0.333)(10.01,0.331)(10.041,0.3) (10.041,-0.3)(10.01,-0.331)(10,-0.333)(9.5,-0.4)(6,-1.5)(4,-0.4)(0.7,-2.9)(0.15,-4.7)};
  \draw[thick] plot[smooth] coordinates {(0.15,4.7)(0.7,2.9)(4,0.4)(6,1.5)(9.5,0.4)(10,0.333)(10.01,0.331)(10.041,0.3)};
  \draw[thick] plot[smooth] coordinates {(10.041,-0.3)(10.01,-0.331)(10,-0.333)(9.5,-0.4)(6,-1.5)(4,-0.4)(0.7,-2.9)(0.15,-4.7)};
  \draw[thick,fill=lightgray] (0,0) ellipse (0.4 and 5);
  \draw[thick] (0,0) ellipse (0.1 and 0.333);
  \draw[thick,fill=gray!15] (10.041,0) ellipse (0.1 and 0.333);
  \draw[thick] plot[smooth] coordinates {(0,2)(2,1.1)(4,-0.1)(6,-0.7)(9,0.25)(10,0.15)};
  \draw[thick,->] (-2,0)--(12,0) node[right,above]{\normalsize$t$};
  \draw[thick,dashed](0,0.333)--(10,0.333);
  \draw[thick,dashed](0,-0.333)--(10,-0.333);
  \node [black] at (0,2) {\textbullet};
  \draw[->,thick](4,-3)node[right]{\normalsize$\lambda$}--(2.5,-0.4);
  \draw[->,thick](3,3)node[right]{\normalsize$(0,e(0))$}--(0.07,2.07);
  \draw[->,thick](9,3)node[right]{\normalsize$\varphi(t)^{-1}$}--(7,1.4);
\end{tikzpicture}
\vspace*{-1mm}
\caption{Error evolution in a funnel $\mathcal F_{\varphi}$ with boundary $\varphi(t)^{-1}$.}
\label{Fig:funnel}
\end{figure}
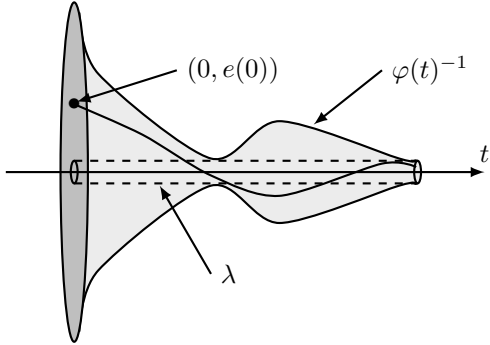
\vspace*{-2mm}
Each performance funnel $\mathcal{F}_{\varphi}$ with $\varphi\in\Phi_r$ is
bounded away from zero as boundedness of $\varphi$ implies that there exists $\lambda>0$ such that $1/\varphi(t)\geq\lambda$ for all $t \ge 0$. The
funnel boundary is not necessarily monotonically decreasing, which might be advantageous in several applications. There are situations where widening the funnel over some later time interval might be beneficial, for instance in the presence of periodic disturbances or strongly varying reference signals. A variety of different funnel boundaries are possible, see e.g.~\cite[Sec.~3.2]{Ilch13}.

\subsection{Organization of the present paper}\label{Ssec:Orga}

In Section~\ref{Sec:TrackAss} we discuss the crucial assumptions in our framework for tracking uncertain non-minimum phase systems. These assumptions lead to the construction of a new output, with respect to which the system has a higher relative degree than~$r$, but the unstable part of the internal dynamics is eliminated. In Section~\ref{Sec:ContrStruc} the controller design is presented, which is based on the funnel controller developed in the recent work~\cite{BergHoan18}. The necessary redefinition of the reference signal is discussed as well and incorporated in the controller. Feasibility of the control is proved in Theorem~\ref{Thm:fun-con}. In Section~\ref{Sec:TransErr} we calculate a bound for the original tracking error, which can be adjusted to be as small as desired by an appropriate choice of the design parameters. The developed controller is then illustrated by a simulation in Section~\ref{Sec:Sim} and some conclusions are given in Section~\ref{Sec:Concl}.
%
\section{Trackability assumptions}\label{Sec:TrackAss}
%

It is revealed in~\cite{GopaHedr93} that for tracking non-minimum systems certain \emph{trackability assumptions} are necessary. In the following, we state the assumptions that are used in the present paper. We stress that these assumptions are much milder than the trackability assumption used in~\cite{GopaHedr93}, which essentially states that the equation
\[
    \dot \eta(t) = Q\eta(t) + P y_{\rm ref}(t)
\]
must have a bounded solution~$\eta:\R_{\ge 0}\to\R^{n-rm}$ for the given reference trajectory $y_{\rm ref}:\R_{\ge 0}\to\R^m$. Here, roughly speaking, we only require this for the non-hyperbolic part of the above equation. We make the following assumptions:
\begin{enumerate}
  \item[\textbf{(A1)}] There exists $T\in\Gl_{n-rm}(\R)$ and $\ell\in\N$ such that
  \begin{equation}\label{eq:decompQ}
    TQT^{-1} = \begin{bmatrix} \hat Q_1 & \hat Q_2 \\ 0&\tilde Q\end{bmatrix},\quad TP = \begin{bmatrix} \hat P\\ \tilde P\end{bmatrix},
  \end{equation}
  where $\hat Q_1\in\R^{k\times k}$, $\hat Q_2\in\R^{k\times \ell m}$, $\tilde Q\in\R^{\ell m\times \ell m}$, $\hat P\in\R^{k\times m}$. $\tilde P\in\R^{\ell m\times m}$, $k=n-rm-\ell m\ge 0$ with $\sigma(\hat Q_1)\subseteq\C_-$ and
  \begin{equation}\label{eq:inv-tildeQP}
     [\tilde P, \tilde Q \tilde P, \ldots, \tilde Q^{\ell-1} \tilde P] \in \Gl_{\ell m}(\R),
  \end{equation}
  and the disturbance satisfies
  \begin{equation}\label{eq:ass-dist-tildeP}
    \forall\, t\ge 0:\ [0, I_{\ell m}] T d_\eta(t) \in \im \tilde P.
  \end{equation}\vspace{-6mm}
  \item[\textbf{(A2)}] Let $y_{\rm ref}\in\cW^{r-1,\infty}(\R_{\ge 0}\to\R^m)$ be a given reference signal and $W\in\Gl_{\ell m}(\R)$ be such that
  \[
    W\tilde Q W^{-1} = \begin{bmatrix} Q_1&0&0\\ 0&Q_2&0\\ 0&0&Q_3\end{bmatrix},\quad W\tilde P = \begin{bmatrix} P_1\\ P_2\\ P_3\end{bmatrix},
  \]
  where $Q_j\in\R^{k_j\times k_j}$, $j=1,2,3$, and $\sigma(Q_1) \subseteq \C_-$, $\sigma(Q_2)\subseteq\C_+$ and $\sigma(Q_3)\subseteq {\rm i}\R$. Then the equation
  \[
    \dot \eta_3(t) = Q_3 \eta_3(t) + P_3 y_{\rm ref}(t),\quad \eta_3(0) = 0
  \]
  has a bounded solution $\eta_3:\R_{\ge 0}\to\R^{k_3}$.
\end{enumerate}
We will frequently choose the smallest~$\ell$ such that~(A1) is satisfied.

\begin{Rem}\ (i)\ We like to give a motivation for assumption~(A1). Basically it states that the unstable part of the matrix~$Q$ in the Byrnes-Isidori form~\eqref{eq:BIF} is completely contained in the matrix~$\tilde Q$ which, together with~$\tilde P$, satisfies the condition~\eqref{eq:inv-tildeQP} that will be explained in more detail later. Note that~$\tilde Q$ may contain some of the stable eigenvalues of~$Q$. Furthermore, we stress that the zero block in $TQT^{-1}$ in~\eqref{eq:decompQ} imposes an additional condition on~$Q$ which is not automatically satisfied, because~$\tilde Q$ must be of size $\ell m\times \ell m$, where~$m$ is given. As an example where such a decomposition is not possible consider $Q=\begin{smallbmatrix} -1 & 1 & 0\\ -1 & -1 & 0\\ 0 & 0 & 1\end{smallbmatrix}$ and $m=2$, where the eigenvalues of~$Q$ are given by~$\{-1, -1, 1\}$. Then the only possible choice for~$\ell$ would be~$\ell=1$, but a decomposition~\eqref{eq:decompQ} is not available in this case.\\
(ii)\ In the case of single-input, single-output systems we have $m=1$ and hence in condition~(A1) it is always possible to find a decomposition~\eqref{eq:decompQ} with $\sigma(\tilde Q)\subseteq \overline{\C_+}$ for some~$\ell\in\N$. However, in order to satisfy the invertibility condition~\eqref{eq:inv-tildeQP} in may be helpful to choose a larger~$\ell$, but then a decomposition~\eqref{eq:decompQ} may not necessarily exist, cf.\ the example given in~(i) above.\\
(iii)\ If the internal dynamics of~\eqref{eq:ABC} have a hyperbolic equilibrium, i.e., $\sigma(Q)\cap {\rm i}\R = \emptyset$ it follows that $k_3=0$. Therefore, assumption~(A2) is always satisfied in this case.\\
(iv)\ One may wonder whether, instead of assuming~(A1) and~(A2), it could be possible to simply assume that the pair~$(A,B)$ is stabilizable and a apply a feedback $u(t) = Fx(t) + v(t)$, where~$v$ is the new input, so that $\sigma(A+BF)\subseteq \C_-$. However, this does not necessarily result in a system which is minimum phase. As a counterexample consider~\eqref{eq:ABC} with $A=\begin{smallbmatrix} 0&1\\ 1&1\end{smallbmatrix}$, $B=\begin{smallbmatrix} 1\\ 0\end{smallbmatrix}$, $C=[1,0]$ and $d(\cdot)=0$. Then assumption~(A1) is satisfied and the system is controllable. However, for any (stabilizing) feedback $u(t) = f_1 x_1(t) + f_2 x_2(t) + v(t)$ the resulting system
        \[
            \begin{pmatrix} \dot y(t)\\ \dot\eta(t)\end{pmatrix} = \begin{bmatrix} f_1&1+f_2\\ 1&1\end{bmatrix} \begin{pmatrix} y(t)\\ \eta(t)\end{pmatrix} + \begin{bmatrix} 1\\ 0\end{bmatrix} v(t)
        \]
        is in Byrnes-Isidori form with $y=x_1$ and $\eta=x_2$, but is not minimum phase since $Q=1$. Thus, a controllable system cannot be rendered minimum phase by a suitable choice of feedback in general.
\end{Rem}

In the following, choose the smallest~$\ell\in\N$ such that assumption~(A1) is satisfied. With the decomposition of~$Q$ as in~\eqref{eq:decompQ} we may further transform the system from~\eqref{eq:BIF} using $T\eta = (\eta_1^\top, \eta_2^\top)^\top$ with $\eta_1:\R_{\ge 0}\to\R^{n-rm-\ell m}$, $\eta_2:\R_{\ge 0}\to\R^{\ell m}$ into
\begin{equation}\label{eq:BIF-refined}
\begin{aligned}
    y^{(r)}(t) &= \sum_{i=1}^r R_i y^{(i-1)}(t) \!+\! S_1 \eta_1(t) \!+\! S_2 \eta_2(t) \!+\! \Gamma u(t) + d_r(t),\\
    \dot \eta_1(t) &= \hat Q_1 \eta_1(t) + \hat Q_2 \eta_2(t) + \hat P y(t) + d_{\eta_1}(t),\\
    \dot \eta_2(t) &= \phantom{\hat Q_1 \eta_1(t) +}\ \ \, \tilde Q \eta_2(t) + \tilde P y(t) + d_{\eta_2}(t),
\end{aligned}
\end{equation}
where $[S_1, S_2] = ST^{-1}$ and $(d_{\eta_1}^\top, d_{\eta_2}^\top)^\top = T d_{\eta}$.

Based on this, we define a new output for system~\eqref{eq:ABC}. Invoking~\eqref{eq:inv-tildeQP}, set
\begin{equation}\label{eq:defK}
    K:=[0,\ldots,0,\Gamma^{-1}] [\tilde P, \tilde Q \tilde P, \ldots, \tilde Q^{\ell-1} \tilde P]^{-1} \in\R^{m\times \ell m}
\end{equation}
and observe that this implies
\[
    K\tilde P = K\tilde Q \tilde P = \ldots = K\tilde Q^{\ell-2} \tilde P = 0,\quad K\tilde Q^{\ell-1}\tilde P = \Gamma^{-1}.
\]
It is a straightforward calculation that the condition~\eqref{eq:ass-dist-tildeP} on the disturbance is equivalent to $K d_{\eta_2}(\cdot) = K \tilde Q d_{\eta_2}(\cdot) = \ldots = K \tilde Q^{\ell-2} d_{\eta_2}(\cdot) = 0$. Therefore, the linear system $\dot z(t) = \tilde Q z(t) + \tilde P \tilde u(t) + d_{\eta_2}(t)$, $\tilde y(t) = K z(t)$ has strict relative degree~$\ell$. In view of~\eqref{eq:ass-dist-tildeP}, let $\delta:\R_{\ge 0}\to\R^{m}$ be such that $\tilde P \delta(t) = d_{\eta_2}(t)$ for all $t\ge 0$; it can be calculated that~$\delta = \Gamma K \tilde Q^{\ell-1} d_{\eta_2}$ is uniquely determined. For later use, we record that it follows from~\cite[Lem.~3.5]{IlchRyan07} that
\begin{equation}\label{eq:cond-ker-KQ}
    \begin{smallbmatrix} K\\ K\tilde Q\\ \vdots\\ K\tilde Q^{\ell -1}\end{smallbmatrix} \in\Gl_{\ell m}(\R).
\end{equation}
As new output for system~\eqref{eq:ABC} we now define
\begin{equation}\label{eq:new-output}
    y_{\rm new}(t) := K \eta_2(t).
\end{equation}
We show that system~\eqref{eq:ABC} with the new output as in~\eqref{eq:new-output} has strict relative degree~$r+\ell$. Assume that $\delta\in\cW^{r,\infty}(\R_{\ge 0}\to\R^m)$ and observe that
\[
    \begin{pmatrix} y_{\rm new}(t) \\ \dot y_{\rm new}(t)\\ \vdots\\ y_{\rm new}^{(\ell-1)}(t)\\ y_{\rm new}^{(\ell)}(t)\end{pmatrix} = \begin{bmatrix} K\\ K\tilde Q\\ \vdots\\ K\tilde Q^{\ell -1}\\ K\tilde Q^\ell\end{bmatrix} \eta_2(t) + \begin{bmatrix} 0\\ \vdots\\ 0\\ \Gamma^{-1}\end{bmatrix} \big(y(t) + \delta(t)\big),
\]
hence it follows from~\eqref{eq:cond-ker-KQ} that there exist $F_1,\ldots, F_\ell\in\R^{\ell m\times m}$ such that
\begin{align}
  \eta_2(t) &= \sum\nolimits_{i=1}^\ell F_i \, y_{\rm new}^{(i-1)}(t),\notag\\
  y(t) &= \Gamma y_{\rm new}^{(\ell)}(t) + \sum\nolimits_{i=1}^\ell \Gamma K\tilde Q^\ell F_i \, y_{\rm new}^{(i-1)}(t) - \delta(t).\label{eq:y-ynew}
\end{align}
Therefore, invoking~\eqref{eq:BIF-refined} it follows that
\begin{align*}
    y_{\rm new}^{(r+\ell)}(t) \!&=\! -\!\sum_{i=1}^\ell K\tilde Q^\ell\! F_i y_{\rm new}^{(r+i-1)}(t) \!+\! \sum_{i=1}^r \Gamma^{-1}\! R_i \Gamma y_{\rm new}^{(\ell+i-1)}(t) \\
    &\quad + \sum_{i=1}^r \sum_{j=1}^\ell \Gamma^{-1} R_i \Gamma K\tilde Q^\ell F_j \,y_{\rm new}^{(i+j-2)}(t)\\
    &\quad +  \sum_{i=1}^\ell \Gamma^{-1} S_2 F_i \, y_{\rm new}^{(i-1)}(t) + S_1 \eta_1(t) + u(t) \\
    &\quad+ \sum_{i=1}^r \Gamma^{-1} R_i \delta^{(i-1)}(t) + \Gamma^{-1} \big(\delta^{(r)}(t) + d_r(t)\big)
\end{align*}
and by some straightforward simplification we obtain
\begin{equation}\label{eq:BIF-final}
\begin{aligned}
    y_{\rm new}^{(r+\ell)}(t) &= \sum_{i=1}^{r+\ell} \hat R_i\, y_{\rm new}^{(i-1)}(t) + S_1 \eta_1(t) + u(t)\\
      &\quad   + \sum_{i=1}^r D_i \delta^{(i-1)}(t) + \Gamma^{-1} \big(\delta^{(r)}(t) + d_r(t)\big),\\
    \dot \eta_1(t) &= \sum_{i=1}^{\ell+1} \hat P_i\, y_{\rm new}^{(i-1)}(t) + \hat Q_1 \eta_1(t) + d_{\eta_1}(t),
\end{aligned}
\end{equation}
for some $\hat R_i, D_i\in\R^{m\times m}$, $i=1,\ldots,r+\ell$ and $\hat P_j\in\R^{(n-rm-\ell m)\times m}$, $j=1,\ldots,\ell+1$. Note that the unstable part of the internal dynamics of~\eqref{eq:ABC}, represented by~$Q_2$ and~$Q_3$, has been completely removed in~\eqref{eq:BIF-final} by using the new output as in~\eqref{eq:new-output}.

\begin{Rem}
  The determination of the new output~\eqref{eq:new-output} is related to finding a so called \emph{flat output} for the subsystem $\dot \eta_2(t) = \tilde Q\eta_2(t) + \tilde P y(t) + d_{\eta_2}(t)$ of the internal dynamics as represented in~\eqref{eq:BIF-refined}, where~$y$ is viewed as the input of this subsystem. Recall that all state and input variables can be parameterized in terms of a flat output, if it exists, see e.g.~\cite{FlieLevi95}. While for linear systems as discussed here this is straightforward, cf.~\eqref{eq:y-ynew}, appropriate results from the theory of differentially flat systems may be helpful for an extension of the results derived in the present paper to nonlinear systems.
\end{Rem}

%
\section{Controller design and feasibility}\label{Sec:ContrStruc}
%

In this section, we propose a novel and simple funnel controller which achieves the control objective. To this end, we will use the recently developed funnel controller from~\cite{BergHoan18} and apply it to the system~\eqref{eq:ABC} with new output~\eqref{eq:new-output}. In order for this to work, since the output has been redefined, tracking requires an appropriate redefinition of the reference signal as well, so that in the end the original output tracks the original reference trajectory with the desired behavior. By the construction of the new output in~\eqref{eq:new-output}, the new reference signal is generated by the corresponding subsystem of~\eqref{eq:BIF-refined} when the original reference signal is inserted for the original output and the disturbance (which is unknown) is ignored, i.e.,
\begin{equation}\label{eq:new-ref}
\begin{aligned}
    \dot \eta_{2,\rm ref}(t) &= \tilde Q \eta_{2,\rm ref}(t) + \tilde P y_{\rm ref}(t),\quad \eta_{2,\rm ref}(0)=\eta_{2,\rm ref}^0,\\
    \hat y_{\rm ref}(t) &= K \eta_{2,\rm ref}(t).
\end{aligned}
\end{equation}
We show in the following that if $y_{\rm ref}\in\mathcal{W}^{r-1,\infty}(\R_{\ge 0}\to\R^m)$, then by assumption~(A2) and an appropriate choice of the initial value~$\eta_{2,\rm ref}^0$ we may achieve that the derivatives of the new reference signal $\hat y_{\rm ref}^{(i)}$ are bounded for $i=0,\ldots,r+\ell$.

\begin{Lem}\label{Lem:new-ref}
Let $y_{\rm ref}\in\mathcal{W}^{r-1,\infty}(\R_{\ge 0}\to\R^m)$, assume that~(A2) holds and use the notation given there. If
\begin{equation}\label{eq:eta-ref-0}
    \eta_{2,\rm ref}^0 = W^{-1} \begin{bmatrix} 0_{k_1\times k_2}\\ -I_{k_2}\\ 0_{k_3\times k_2}\end{bmatrix} \int_0^\infty e^{-Q_2 s} P_2 y_{\rm ref}(s) \ds{s},
\end{equation}
then the initial value problem~\eqref{eq:new-ref} has a unique global solution such that $\hat y_{\rm ref}\in\mathcal{W}^{r+\ell,\infty}(\R_{\ge 0}\to\R^m)$.
\end{Lem}
\begin{proof} Recall the decomposition of $W \tilde Q W^{-1}$ and $W\tilde P$ from~(A2). First note that~\eqref{eq:eta-ref-0} is well-defined since~$y_{\rm ref}$ is bounded and $\sigma(-Q_2)\subseteq\C_-$. Furthermore, the initial value problem
\[
    \dot z_1(t) = Q_1 z_1(t) + P_1 y_{\rm ref}(t),\quad z_1(0) = 0,
\]
has a unique global solution which satisfies $z_1\in\mathcal{W}^{r,\infty}(\R_{\ge 0}\to\R^{\ell m -k_2-k_3})$. It follows from~(A2) that
\[
    \dot z_3(t) = Q_3 z_3(t) + P_3 y_{\rm ref}(t),\quad z_3(0) = 0,
\]
has a unique global solution which is bounded. Successively taking the derivative of~$z_3$ and evaluating the differential equation gives that $z_3\in\mathcal{W}^{r,\infty}(\R_{\ge 0}\to\R^{k_3})$. Finally, to show that
\begin{align*}
    \dot z_2(t) &= Q_2 z_2(t) + P_2 y_{\rm ref}(t),\\ z_2(0) &= -\int_0^\infty e^{-Q_2 s} P_2 y_{\rm ref}(s) \ds{s},
\end{align*}
has a bounded and unique global solution, although $\sigma(Q_2)\subseteq\C_+$, we use the following straightforward result for linear systems: For $A\in\R^{n\times n}$ with $\sigma(A)\subseteq\C_+$ and $B\in\R^{n\times m}$, $u\in\cL^\infty(\R_{\ge 0}\to\R^m)$, $x^0\in\R^n$ there exists $x\in\cW^{1,\infty}(\R_{\ge 0}\to\R^n)$ which solves $\dot x(t) = Ax(t) + B u(t)$ with $x(0)=x^0$ if, and only if,
$x^0 + \int_0^\infty e^{-A s} B u(s) \ds{s} = 0$. Therefore, we infer $z_2\in\mathcal{W}^{r,\infty}(\R_{\ge 0}\to\R^{k_2})$.

Now, we have that $\eta_{2,\rm ref} = W^{-1} (z_1^\top, z_2^\top, z_3^\top)^\top$ and, similar to Section~\ref{Sec:TrackAss}, we may derive that
\[
    \hat y_{\rm ref}^{(\ell)}(t) = K \tilde Q^\ell \eta_{2,\rm ref}(t) + \Gamma^{-1} y_{\rm ref}(t),
\]
and hence it follows that $\hat y_{\rm ref}\in\mathcal{W}^{r+\ell,\infty}(\R_{\ge 0}\to\R^m)$.
\end{proof}

The computation of $\eta_{2,\rm ref}^0$ in~\eqref{eq:eta-ref-0} requires the knowledge of~$y_{\rm ref}(t)$ for all $t\ge 0$, hence it is an acausal problem which may impose a challenge in applications. However, a large class of reference signals can be generated by linear \emph{exosystems} (as in linear regulator problems, cf.~\cite{Fran77}) of the form
\begin{equation}\label{eq:exo}
    \dot w(t) = A_e w(t),\quad y_{\rm ref}(t) = C_e w(t),\quad w(0)=w^0,
\end{equation}
where the parameters~$A_e\in\R^{k\times k}, C_e\in\R^{m\times k}$ and $w^0\in\R^k$ are known, and $\sigma(A_e)\subseteq \overline{\C_-}$ so that any eigenvalue $\lambda\in\sigma(A_e)\cap i\R$ is semisimple; this guarantees $y_{\rm ref}\in\cW^{r-1,\infty}(\R_{\ge 0}\to\R^m)$ as required in~(A2). In fact, by a Fourier series argument, on any interval of interest we may approximate any given~$y_{\rm ref}$ arbitrarily good by a exosystem~\eqref{eq:exo}, when~$k$ is large enough. We show that then~$\eta^0_{2,\rm ref}$ can be computed using the solution of a certain Sylvester equation.

\begin{Lem}\label{Lem:exo-Sylv}
Consider the exosystem~\eqref{eq:exo} with output~$y_{\rm ref}$, assume that~(A2) holds and use the notation given there. Then the Sylvester equation
\begin{equation}\label{eq:Sylv}
    Q_2 X - X A_e = P_2 C_e
\end{equation}
has a unique solution $X\in\R^{k_2\times k}$ and $\eta_{2,\rm ref}^0$ as in~\eqref{eq:eta-ref-0} is given by
\begin{equation}\label{eq:eta20-exo}
    \eta^0_{2,\rm ref} = W^{-1} \begin{bmatrix} 0_{k_1\times k_2}\\ -I_{k_2}\\ 0_{k_3\times k_2}\end{bmatrix} X w^0.
\end{equation}
\end{Lem}
\begin{proof} Since $\sigma(A_e)\cap\sigma(Q_2) = \emptyset$ it follows from~\cite[Thm.~1]{Chu87} that the Sylvester equation~\eqref{eq:Sylv} always has a unique solution $X\in\R^{k_2\times k}$. Now, in view of $y_{\rm ref}(t) = C_e e^{A_e t} w^0$, using integration by parts we find that
\begin{multline*}
    \int_0^\infty - Q_2 e^{-Q_2 s} P_2 C_e e^{A_e s} {\rm d}s + \int_0^\infty  e^{-Q_2 s} P_2 C_e e^{A_e s} A_e {\rm d}s\ \ \\
    = \left[ e^{-Q_2 s} P_2 C_e e^{A_e s} \right]_0^\infty = - P_2 C_e,
\end{multline*}
where the latter equality follows since $s\mapsto  e^{A_e s}$ is bounded and $\lim_{s\to\infty} e^{-Q_2 s} = 0$.  This equation is of the form~\eqref{eq:Sylv}, and uniqueness of~$X$ gives
\[
    X = \int_0^\infty e^{-Q_2 s} P_2 C_e e^{A_e s} {\rm d}s.
\]
Invoking~\eqref{eq:eta-ref-0}, this implies~\eqref{eq:eta20-exo}.
\end{proof}

Note that the Sylvester equation~\eqref{eq:Sylv} can be solved in MATLAB using the \textsc{sylvester} command for instance.

\captionsetup[subfloat]{labelformat=empty}
\begin{figure*}[h!t]
\centering
\resizebox{13cm}{!}{
  \begin{tikzpicture}[very thick,scale=0.7,node distance = 9ex, box/.style={fill=white,rectangle, draw=black}, blackdot/.style={inner sep = 0, minimum size=3pt,shape=circle,fill,draw=black},blackdotsmall/.style={inner sep = 0, minimum size=0.1pt,shape=circle,fill,draw=black},plus/.style={fill=white,circle,inner sep = 0,very thick,draw},metabox/.style={inner sep = 3ex,rectangle,draw,dotted,fill=gray!20!white}]

    \node (sys)     [box,minimum size=10ex,xshift=-1ex]  {$\begin{aligned}
    \dot x(t) &= Ax(t) + Bu(t) +d(t)\\
    y(t) &= Cx(t)
    \end{aligned}
    $};
    \node (dist) [box, above of = sys, yshift=2ex,minimum size=4ex] {$d\in\cL^{\infty}(\R_{\ge 0}\to\R^n)$};
    \node(BIF) [box, right of = sys,xshift=47ex,minimum size=8ex] {$\begin{aligned}
    y^{(r)}(t) &= \sum_{i=1}^r R_i y^{(i-1)}(t) + S \eta(t) + \Gamma u(t) + d_r(t)\\
    \dot \eta(t) &= P y(t) + Q\eta(t) + d_\eta(t)
\end{aligned}$};
    \node(QP) [box, below of = BIF,yshift=-10ex,minimum size=8ex] {$\begin{aligned}
    &\text{obtain $(\tilde Q, \tilde P)$ with}\\
    &\dot \eta_2(t) = \tilde Q \eta_2(t) + \tilde P y(t)
\end{aligned}$};
    \node(ynew) [box, left of = QP,xshift=-47ex,minimum size=4ex] {$y_{\rm new}(t) = K \eta_2(t)$};
    \node(eta) [box, below of = QP,yshift=-8ex,minimum size=8ex] {$\begin{aligned}
    \dot \eta_{2,\rm ref}(t) &= \tilde Q \eta_{2,\rm ref}(t) + \tilde P y_{\rm ref}(t),\quad \eta_{2,\rm ref}(0)=\eta_{2,\rm ref}^0\\
    \hat y_{\rm ref}(t) &= K \eta_{2,\rm ref}(t)
\end{aligned}$};
    \node(yref) [box, below of = eta, yshift=-1ex,minimum size=4ex] {$y_{\rm ref}\in\cW^{r-1,\infty}(\R_{\ge 0}\to\R^m)$};
    \node(FC) [box, left of = eta, xshift=-47ex,minimum size=8ex] {Controller~\eqref{eq:fun-con}};
    \node(phi) [box, below of = FC, yshift=-1ex,minimum size=4ex] {$\varphi_i\in\Phi_{r+\ell-i},\ i=0,\ldots,r+\ell-1$};

    \node (c1)  [left of = ynew, xshift=-8ex, yshift=1ex] {};
    \node (c2)  [left of = ynew, xshift=-10ex, yshift=-1ex] {};

    \draw[->] (sys) -- (BIF) node[midway,above] {Compute~\eqref{eq:BIF}};
    \draw[->] (BIF) -- (QP) node[midway,right] {Check (A1)--(A3)};
    \draw[->] (QP) -- (ynew) node[midway,above] {Compute~$K$ in~\eqref{eq:defK}};
    \draw[->] (ynew) -- (FC) {};
    \draw[->] (QP) -- (eta) node[midway,right] {Compute~$\eta_{2,\rm ref}^0$ in~\eqref{eq:eta-ref-0}};
    \draw[->] (yref) -- (eta) {};
    \draw[->] (phi) -- (FC) {};
    \draw[->] (eta) -- (FC) {};
    \draw[->] (dist) -- (sys) {};
    \path[->,shorten >=4pt] (FC.west) edge[bend left] node[midway,left] {apply} (sys.west);
  \end{tikzpicture}
}
\caption{Construction of the funnel controller~\eqref{eq:fun-con} depending on its design parameters.\\[-4mm]}
\label{Fig:con-constr}
\end{figure*}
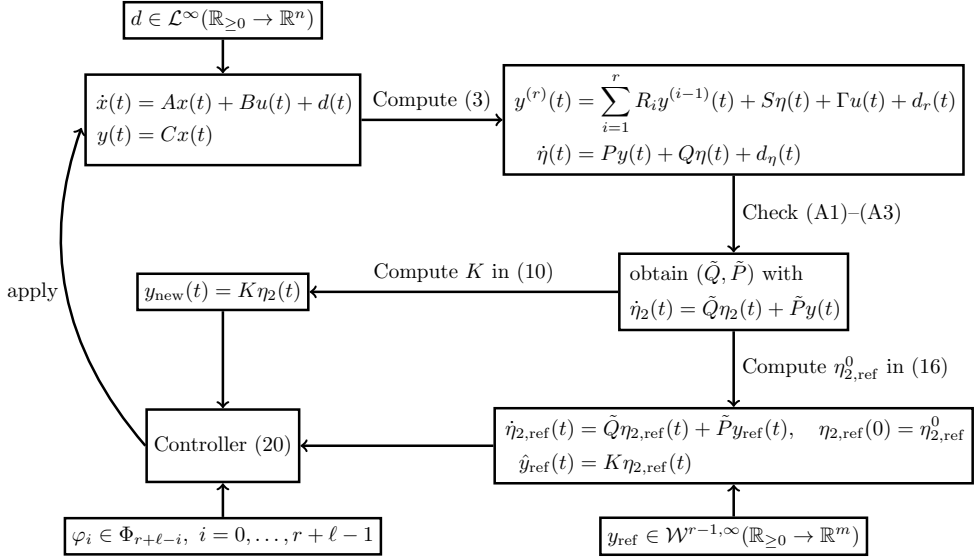

The generator~\eqref{eq:new-ref} of the new reference signal will now be incorporated as a dynamic part into the controller design and the funnel controller from~\cite{BergHoan18} will be applied to system~\eqref{eq:ABC} with new output $y_{\rm new}$ as in~\eqref{eq:new-output}. The final controller design is of the form~\eqref{eq:objcontr} and given by:
\begin{equation}\label{eq:fun-con}
\boxed{\begin{aligned}
\dot \eta_{2,\rm ref}(t) &= \tilde Q \eta_{2,\rm ref}(t) + \tilde P y_{\rm ref}(t),\quad \eta_{2,\rm ref}(0)=\eta_{2,\rm ref}^0,\\
    \hat y_{\rm ref}(t) &= K \eta_{2,\rm ref}(t),\\
e_0(t)&= y_{\rm new}(t) - \hat y_{\rm ref}(t),\\
e_1(t)&=\dot{e}_0(t)+k_0(t)\,e_0(t),\\
e_2(t)&=\dot{e}_1(t)+k_1(t)\,e_1(t),\\
&\, \ \vdots \\
e_{r+\ell-1}(t)&=\dot{e}_{r+\ell-2}(t)+k_{r+\ell-2}(t)\,e_{r+\ell-2}(t),\\
k_i(t)&=\frac{1}{1-\varphi_i(t)^2\|e_i(t)\|^2},\quad i=0,\dots,r\!+\!\ell\!-\!1, \\
u(t) &= -k_{r+\ell-1}(t)\,e_{r+\ell-1}(t),
\end{aligned}
}
\end{equation}
where the initial value $\eta_{2,\rm ref}^0$ is as in~\eqref{eq:eta-ref-0} and the reference signal and funnel functions have the following properties:
\begin{equation}\label{eq:con-ass}
\boxed{
\begin{aligned}
y_{\rm ref}&\in\, \mathcal{W}^{r-1,\infty}(\R_{\ge 0}\rightarrow \R^m),\\
\varphi_0&\in\, \Phi_{r+\ell},\;\;
\varphi_1\in \Phi_{r+\ell-1},\;\ldots,\;\;
\varphi_{r+\ell-1}\in \Phi_{1}.\\
\end{aligned}
}
\end{equation}
The construction of the funnel controller~\eqref{eq:fun-con} is summarized in Fig.~\ref{Fig:con-constr}.

We stress that the derivatives $\dot e_0,\ldots, \dot e_{r+\ell-2}$ which appear in~\eqref{eq:fun-con} only serve as short-hand notations and may be resolved in terms of the virtual tracking error~$e_0$, the funnel functions~$\varphi_i$ and the derivatives of these, cf.~\cite[Rem.~2.1]{BergHoan18}. For~\eqref{eq:fun-con} to be robust, it is necessary that $\dot e_0,\ldots, \dot e_{r+\ell-2}$ can be obtained from measurements independent of the disturbance. However, by~\eqref{eq:y-ynew} the derivatives $y_{\rm new}^{(\ell)},\ldots, y_{\rm new}^{(r+\ell-1)}$ depend on $\delta, \dot \delta,\ldots, \delta^{(r-1)}$. Therefore, since $r\ge 1$, we need to require that $\delta = 0$ and hence
\begin{enumerate}
  \item[\textbf{(A3)}] $0 = d_{\eta_2}(\cdot) = [0, I_{\ell m}] T [0, I_{n-rm}] U d(\cdot)$.
\end{enumerate}
This assumption essentially means that the unstable part of the internal dynamics of~\eqref{eq:ABC} is not affected by the disturbance. The controller structure is depicted in Fig.~\ref{Fig:controller}.

\captionsetup[subfloat]{labelformat=empty}
\begin{figure*}[h!t]
\centering
\resizebox{11.5cm}{!}{
  \begin{tikzpicture}[very thick,scale=0.7,node distance = 9ex, box/.style={fill=white,rectangle, draw=black}, blackdot/.style={inner sep = 0, minimum size=3pt,shape=circle,fill,draw=black},blackdotsmall/.style={inner sep = 0, minimum size=0.1pt,shape=circle,fill,draw=black},plus/.style={fill=white,circle,inner sep = 0,very thick,draw},metabox/.style={inner sep = 3ex,rectangle,draw,dotted,fill=gray!20!white}]
 \begin{scope}[scale=0.5]
    \node (sys)     [box,minimum size=9ex,xshift=-1ex]  {$\begin{aligned}
    \dot x(t) &= Ax(t) + Bu(t) + d(t)\\
    y(t) &= Cx(t)\\
    y_{\rm new}(t) &= K\eta_2(t)
\end{aligned}
$};

    \node(FC) [box, below of = sys,yshift=-8ex,minimum size=10ex] {$\begin{aligned} &\text{Funnel Controller}\\ &\text{from~\cite{BergHoan18}}\end{aligned}$};
    \node(fork1) [plus, right of = FC, xshift=15ex] {$+$};
    \node(ref) [box, right of = sys,minimum size=10ex, xshift=37.5ex] {$\begin{aligned}
    \dot \eta_{2,\rm ref}(t) &= \tilde Q \eta_{2,\rm ref}(t) + \tilde P y_{\rm ref}(t)\\
    \hat y_{\rm ref}(t) &= K \eta_{2,\rm ref}(t)\\
    \eta_{2,\rm ref}(0) &=\eta_{2,\rm ref}^0\quad \text{as in~\eqref{eq:eta-ref-0}}
\end{aligned}$};
    \node(fork2) [plus, above of = fork1, yshift=20ex] {$+$};
    \node(forku) [minimum size=0pt, inner sep = 0pt, left of = FC, xshift=-12.5ex] {};
    \node(refin) [minimum size=0pt, inner sep = 0pt, above of = ref, yshift=8ex] {};
    \node(refout) [minimum size=0pt, inner sep = 0pt, above of = fork2, yshift=0ex] {};
    \node (distin) [minimum size=0pt, inner sep = 0pt, above of = sys, yshift=4ex, xshift=-5ex] {};
    \node (distout) [minimum size=0pt, inner sep = 0pt, below of = distin, yshift=1.7ex] {};
    \node (yout) [minimum size=0pt, inner sep = 0pt, right of = distout, xshift=2ex] {};

    \draw[->] (sys) -| (fork1) node[pos=0.3,above] {$y_{\rm new}(t)$} node[pos=0.9,left] {$+$};
    \draw[->] (refin) -- (ref) node[pos=0.3,right] {$y_{\rm ref}(t)$};
    \draw[->] (ref) |- (fork1) node[pos=0.2,right] {$\hat y_{\rm ref}(t)$} node[pos=0.92,above] {$-$};
    \draw[->] (fork1) -- (FC) node[midway,above] {$e_0(t)$};
    \draw (FC) -- (forku.west) {};
    \draw[->] (forku.south) |- (sys) node[pos=0.7,above] {$u(t)$};
    \draw[->] (yout) |- (fork2) node[pos=0.3,left] {$y(t)$} node[pos=0.92,above] {$+$};
    \draw[->] (refin) |- (fork2) node[pos=0.92,above] {$-$};
    \draw[->] (fork2) -- (refout) node[midway,right] {$e(t)$};
    \draw[->] (distin) -- (distout) node[midway,left] {$d(t)$};

\end{scope}
\begin{pgfonlayer}{background}
      \fill[lightgray!20] (5.5,2) rectangle (14,-5.5);
       \fill[lightgray!20] (-3.5,-2) rectangle (14,-5.5);
       \draw[dotted] (-3.5,-2) -- (5.5,-2) -- (5.5,2) -- (14,2) -- (14,-5.5) -- (-3.5,-5.5) -- (-3.5,-2);
  \end{pgfonlayer}
  \end{tikzpicture}
}
\caption{The funnel controller~\eqref{eq:fun-con}, indicated by the grey box, applied to system~\eqref{eq:ABC} with new output as in~\eqref{eq:new-output}. The controller consists of the generator of the new reference signal~\eqref{eq:new-ref} and the funnel controller developed in~\cite{BergHoan18}.\\[-6mm]}
\label{Fig:controller}
\end{figure*}
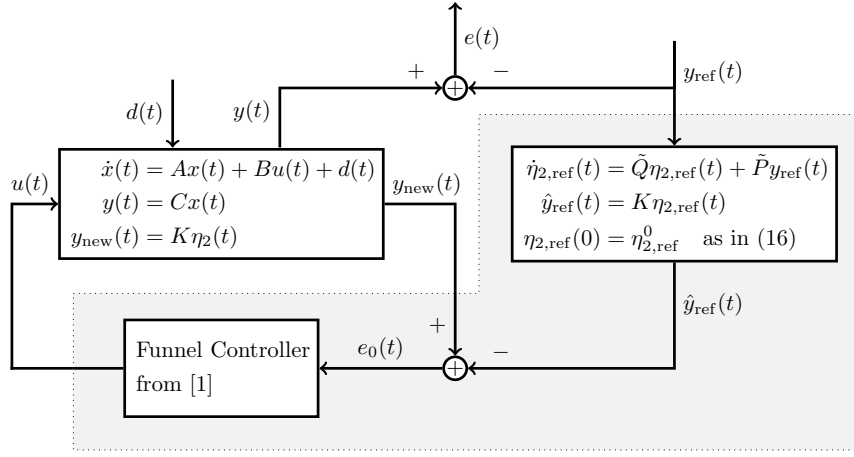

The application of the controller~\eqref{eq:fun-con} to the linear system~\eqref{eq:ABC} with new output as in~\eqref{eq:new-output} results in a nonlinear and time-varying closed-loop differential equation in general, defined on a proper subset of $\R_{\ge 0}\times\R^{n+\ell m}$ due to the poles introduced by~$k_i$. Hence, some care must be exercised with the existence of a \emph{solution} of~\eqref{eq:ABC},~\eqref{eq:fun-con}, by which we mean a weakly differentiable function $(x,\eta_{2,\rm ref}):[0,\omega)\to\R^{n+\ell m}$, $\omega\in(0,\infty]$, which satisfies the initial conditions and differential equations in~\eqref{eq:ABC},~\eqref{eq:fun-con} for almost all $t\in[0,\omega)$; $(x,\eta_{2,\rm ref})$ is called \emph{maximal}, if it has no right extension that is also a solution.

Concluding this section, we show feasibility of the novel funnel controller design~\eqref{eq:fun-con}, which is one of the main results of the present paper.

\begin{Thm}\label{Thm:fun-con}
Consider a linear system~\eqref{eq:ABC} which satisfies~\eqref{eq:rel-deg} and assumptions~(A1)--(A3). Let $\ell\in\N$ be the smallest number such that~(A1) is satisfied. Further let $y_{\rm ref}, \varphi_0,\ldots, \varphi_{r+\ell-1}$ be as in~\eqref{eq:con-ass} and $x^0\in \R^n$ be an initial value such that $e_0,\ldots, e_{r+\ell-1}$ as defined in~\eqref{eq:fun-con} satisfy
\[
    \varphi_i(0) \|e_i(0)\| < 1\quad \text{for}\ i=0,\ldots,r+\ell-1.
\]
Then the controller~\eqref{eq:fun-con} applied to~\eqref{eq:ABC} yields a closed-loop system which has a unique global solution $(x,\eta_{2,\rm ref}):[0,\infty)\to\R^{n+\ell m}$ with the properties:
\begin{enumerate}
  \item all involved signals $x(\cdot)$, $\eta_{2,\rm ref}(\cdot)$, $u(\cdot)$, $k_0(\cdot), \ldots, k_{r+\ell-1}(\cdot)$ are bounded;
  \item the errors evolve uniformly within the respective performance funnels in the sense
  \begin{equation}\label{eq:error-evlt}
  \begin{aligned}
    \forall\, i=0,\ldots,r+\ell-1\ & \exists\, \eps_i>0\ \forall\, t\ge 0:\\
    & \|e_i(t)\| \le \varphi_i(t)^{-1}-\eps_i.
  \end{aligned}
  \end{equation}
\end{enumerate}
\end{Thm}
\begin{proof} By assumptions~\eqref{eq:rel-deg},~(A1) and (A2) and the calculations made in Section~\ref{Sec:TrackAss} we find that system~\eqref{eq:ABC} with new output~\eqref{eq:new-output} is equivalent to~\eqref{eq:BIF-final} and, in particular, it has strict relative degree $r+\ell$. Since $\sigma(\hat Q_1)\subseteq \C_-$, $\delta=0$ by~(A3) and $d_r, d_{\eta_1}$ are bounded it is straightforward to see that~\eqref{eq:BIF-final} belongs to the system class discussed in~\cite{BergHoan18}. Furthermore, the new reference signal $\hat y_{\rm ref}$ generated by~\eqref{eq:new-ref} satisfies $\hat y_{\rm ref}\in\mathcal{W}^{r+\ell,\infty}(\R_{\ge 0}\to\R^m)$ by Lemma~\ref{Lem:new-ref}. Therefore, we may apply~\cite[Thm.~3.1]{BergHoan18} to~\eqref{eq:ABC} with new output~\eqref{eq:new-output} and new reference signal $\hat y_{\rm ref}$, which implies the statements of the theorem, except for uniqueness of the solution $(x,\eta_{2,\rm ref})$. However, the latter follows from the theory of ordinary differential equations, see e.g.~\cite[\S~10, Thm.~XX]{Walt98}, since the right-hand side of the closed-loop differential equation is measurable and locally integrable in~$t$ and locally Lipschitz in the other variables.
\end{proof}

\begin{Rem}
  We like to emphasize that the controller~\eqref{eq:fun-con} depends on the initial~$\eta_{2,\rm ref}^0$ which must be computed as in~\eqref{eq:eta-ref-0} in order to obtain $\hat y_{\rm ref}\in\cW^{r+l,\infty}(\R_{\ge 0}\to\R^m)$ so that the control is feasible. If a small error is made in this computation, say $\hat \eta_{2,\rm ref}^0 = \eta_{2,\rm ref}^0 + \cE$ is computed, where $\cE\in\R^{\ell m}$ with $\|\cE\|\le \varepsilon$ for some $\varepsilon>0$, then $\hat y_{\rm ref}$ may not be bounded, and hence the controller does not provide a bounded global solution in general (although a global solution still exists). The reason is that the unstable part of the internal dynamics may amplify the error, i.e., we have (with $\tilde Q = Q_2$ and $\tilde P = P_2$ for simplicity), $\eta_2(t) = p(t) + e^{Q_2 t} \cE$, where $p(t) =  e^{Q_2 t} \eta_{2,\rm ref}^0 + \int_0^t e^{Q_2(t-s)} P_2 y_{\rm ref}(s) \ds{s}$ is bounded by Lemma~\ref{Lem:new-ref}, but $t\mapsto e^{Q_2 t} \cE$ grows unbounded. One possibility to resolve this is to recalculate the value of~$\eta_{2,\rm ref}^0$ as in~\eqref{eq:eta-ref-0} at discrete time points $t_n = n T$ for $n\in\N$ and fix $T>0$. If the error $\hat \eta_{2,\rm ref}^n = \eta_{2,\rm ref}^n + \cE_n$ made each time satisfies $\|\cE_n\| \le \varepsilon$, then the correction term is bounded,
  \[
   \forall\,n\in\N\ \forall\, t\in [nT, (n+1)T]:\ \left\| e^{Q_2 (t- nT)} \cE_n\right\| \le e^{\|Q_2\| T} \varepsilon,
  \]
  and hence $\eta_2$ and $\hat y_{\rm ref}$ are globally bounded; note that this is independent of the choice of $T>0$ and $\varepsilon>0$, and the latter does not even need to be known. Therefore, accordingly restarting the funnel controller~\eqref{eq:fun-con} at each time $t_n = nT$ and ensuring that $\varphi_i(t_n) \| e_i(t_n)\| < 1$ is satisfied for $i=0,\ldots,r+\ell-1$, $n\in\N$, guarantees existence of a global bounded solution of the closed-loop system such that~\eqref{eq:error-evlt} is satisfied; this relies on estimating the input by the global bound derived in~\cite[Prop.~3.2]{BergHoan18} on each interval $[nT, (n+1)T]$ using the uniform bound for~$\hat y_{\rm ref}$ from above.
\end{Rem}

We stress that Theorem~\ref{Thm:fun-con} does not provide a bound for the original tracking error~$e(t)$; this will be discussed in the subsequent section.

%
\section{Transient behavior of the original tracking error}\label{Sec:TransErr}
%

In this section we provide a bound for the transient behavior of the tracking error~$e(t)$, which may be calculated a priori and can be adjusted using the funnel functions~$\varphi_0,\ldots,\varphi_\ell$. To obtain a reasonable bound we first need to improve the estimate~\eqref{eq:error-evlt} in the sense that at each time $t\ge 0$ we need to find ``the best'' $\eps_i(t)$ such that $\|e_i(t)\|\le \varphi_i(t)^{-1}-\eps_i(t)$; still ensuring that~$\eps_i(\cdot)$ can be calculated a priori. One possible choice for (constant)~$\eps_i$ is provided in~\cite{BergHoan18}, but this choice is far from being optimal. In the following we derive an improvement of this.

To this end, use the notation and assumptions from Theorem~\ref{Thm:fun-con}, and set $\psi_i(t):=\varphi_i(t)^{-1}$ for all $t\geq 0$ and all $i=0,\ldots,r+\ell-1$. Then, by~\eqref{eq:con-ass}, $\psi_i:\R_{\ge 0}\to\R_{>0}$ is continuously differentiable and $\dot\psi_i$ is bounded, $i=0,\ldots,r+\ell-1$. Consider the initial value problems
\begin{equation}\label{eq:ODE-eps_i}
\begin{aligned}
  \dot \eps_i(t) &= \dot \psi_i(t) - \psi_{i+1}(t) + \frac{\psi_i(t)\big(\psi_i(t) - \eps_i(t)\big)}{2\eps_i(t)},\\
  \eps_i(0) &= \psi_i(0) - \|e_i(0)\|
\end{aligned}
\end{equation}
for $i=0,\ldots,r+\ell-2$. In the following result we show that~\eqref{eq:ODE-eps_i} indeed has a unique global solution and provide some bounds for it.

\begin{Lem}\label{Lem:est-eps_i}
Use the notation and assumptions from Theorem~\ref{Thm:fun-con}. Set
\begin{align*}
    \lambda_i &:=\inf_{t\geq 0}\psi_i(t)>0,\quad i=0,\ldots,r+\ell-1,\\
    \kappa_i &:= \|\psi_{i+1} - \dot \psi_i\|_\infty,\quad i=0,\ldots,r+\ell-2,\\
    \eps_{i,\min} &:= \min\left\{\tfrac{\la_i^2}{2\kappa_i + \|\psi_i\|_\infty}, \psi_i(0) - \|e_i(0)\|\right\} > 0,\\
    \eps_{i,\max} &:= \min\left\{\tfrac{\la_{i+1} \la_i}{\|\psi_i\|_\infty}, \tfrac{\la_i}{2}, \|e_i(0)\|\right\} \ge 0.
\end{align*}
Then, for all $i=0,\ldots,r+\ell-2$, the initial value problem~\eqref{eq:ODE-eps_i} has a unique global solution $\eps_i:\R_{\ge 0}\to\R$ that satisfies
\begin{equation}\label{eq:est-eps_i}
   \forall\, t\ge 0:\ \eps_{i,\min} \le \eps_i(t) \le \psi_i(t) - \eps_{i,\max}.
\end{equation}
\end{Lem}
\begin{proof} Since the right hand side of the differential equation in~\eqref{eq:ODE-eps_i} is measurable and locally integrable in~$t$ and locally Lipschitz in~$\eps_i$ (as a function defined on the relatively open set $\setdef{(t,\eps)\in\R_{\ge 0}\times\R}{\eps>0}$), it follows from the theory of ordinary differential equations, see e.g.~\cite[\S~10, Thm.~XX]{Walt98}, that~\eqref{eq:ODE-eps_i} has a unique maximal solution $\eps_i:[0,\omega)\to\R$ with $\omega\in(0,\infty]$, such that~$\eps_i$ is weakly differentiable and $\eps_i(t)>0$ for all $t\in[0,\omega)$. Furthermore,  the closure of the graph of~$\eps_i$ is not a compact subset of $\setdef{(t,\eps)\in\R_{\ge 0}\times\R}{\eps>0}$. It remains to show $\omega=\infty$ and~\eqref{eq:est-eps_i}.

We first show that $\eps_{i,\min} \le \eps_i(t)$ for all $t\in[0,\omega)$. Seeking a contradiction assume that there exists $t_1\in[0,\omega)$ such that $0<\eps_i(t_1) < \eps_{i,\min}$. Since $\eps_i(0) \ge \eps_{i,\min}$ there exists
\[
    t_0 := \max \setdef{t\in [0,t_1)}{ \eps_i(t) = \eps_{i,\min}},
\]
and we find that $\eps_i(t)\le \eps_{i,\min}$ for all $t\in[t_0,t_1]$. Then it follows that
\begin{align*}
    \dot \eps_i(t) &= \dot \psi_i(t) - \psi_{i+1}(t) + \frac{\psi_i(t)\big(\psi_i(t) - \eps_i(t)\big)}{2\eps_i(t)}\\
    &\ge -\kappa_i + \frac{\la_i^2}{2 \eps_{i,\min}} - \frac{\|\psi_i\|_\infty}{2} \ge 0
\end{align*}
for almost all $t\in [t_0,t_1]$. Therefore,
\[
  \eps_{i,\min} = \eps_i(t_0) \le  \eps_i(t_1) < \eps_{i,\min},
\]
a contradiction.

Now we show that $\eps_i(t) \le \psi_i(t) - \eps_{i,\max}$ for all $t\in[0,\omega)$. Again seeking a contradiction assume that there exists $t_1\in[0,\omega)$ such that $\eps_i(t_1) > \psi_i(t_1) - \eps_{i,\max}$. Since $\eps_i(0) = \psi_i(0) - \|e_i(0)\| \le \psi_i(0) - \eps_{i,\max}$ there exists
\[
    t_0 := \max \setdef{t\in [0,t_1)}{ \eps_i(t) = \psi_i(t) - \eps_{i,\max}}.
\]
Then it follows that $\eps_i(t) \ge \psi_i(t) - \eps_{i,\max} \ge \tfrac{\la_i}{2}$ for all $t\in[t_0,t_1]$ and hence
\begin{align*}
    \dot \eps_i(t) &= \dot \psi_i(t) - \psi_{i+1}(t) + \frac{\psi_i(t)\big(\psi_i(t) - \eps_i(t)\big)}{2\eps_i(t)}\\
    &\le \dot \psi_i(t) - \la_{i+1} + \frac{\|\psi_i\|_\infty\, \eps_{i,\max}}{\la_i} \le   \dot \psi_i(t)
\end{align*}
for almost all $t\in [t_0,t_1]$. Therefore, $\eps_i(t_1) - \eps_i(t_0) \le \psi_i(t_1) - \psi_i(t_0)$ which gives
\[
  \eps_{i,\max} = \psi_i(t_0) - \eps_i(t_0) \le  \psi_i(t_1) - \eps_i(t_1) < \eps_{i,\max},
\]
a contradiction.

To see that $\omega=\infty$, assume $\omega<\infty$ which, invoking $\eps_{i,\min} \le \eps_i(t) \le \psi_i(t) - \eps_{i,\max}$ for all $t\in[0,\omega)$, implies that the closure of the graph of the solution~$\eps_i$ is a compact subset of $\setdef{(t,\eps)\in\R_{\ge 0}\times\R}{\eps>0}$, a contradiction.
\end{proof}

Note that in~\eqref{eq:est-eps_i} it is possible that $\eps_{i,\max}=0$, which is the case if, and only if, $e_i(0)=0$.

\begin{Ex}
  We illustrate a typical situation, where the funnel functions $\psi_i(\cdot) = \varphi_i(\cdot)^{-1}$ are of exponential decreasing form $\psi_i(t) = a_i e^{-b_i t} + c_i$. Here we choose
  \[
    \psi_0(t) = e^{-2t} + 0.1,\quad \psi_1(t)=2 e^{-2t} + 0.02,
  \]
  and $\|e_i(0)\| = \psi_0(0)/2 = 0.55$. The solution~$\eps_0$ of~\eqref{eq:ODE-eps_i} for $i=0$ is depicted in Fig.~\ref{fig:eps}.

\begin{figure}[h!tb]
\hspace{3mm}\includegraphics[trim=20 10 20 10,clip,width=7.5cm]{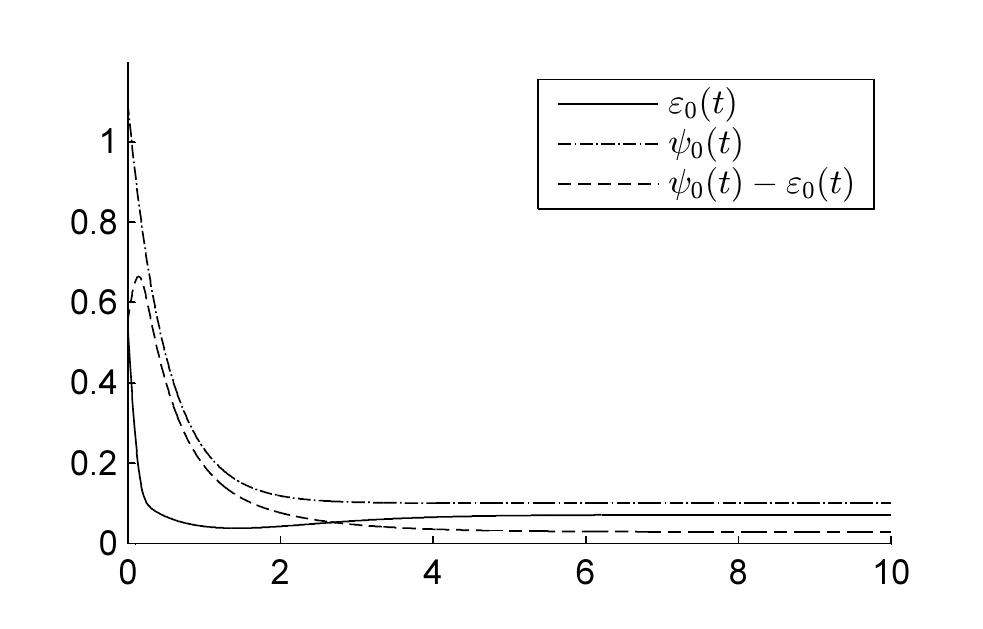}
\vspace{-2mm}
\caption{Solution~$\eps_0$ of~\eqref{eq:ODE-eps_i} for $i=0$.}
\label{fig:eps}
\end{figure}
\end{Ex}

Using the solutions of~\eqref{eq:ODE-eps_i} we may now improve the estimates~\eqref{eq:error-evlt} from Theorem~\ref{Thm:fun-con}.

\begin{Lem}
Use the notation and assumptions from Theorem~\ref{Thm:fun-con}. Let $(x,\eta_{2,\rm ref}):\R_{\ge 0}\to\R^{n+\ell m}$
be the solution of~\eqref{eq:ABC},~\eqref{eq:fun-con} and let $\eps_i:\R_{\ge 0}\to\R$ be the solution of~\eqref{eq:ODE-eps_i} for $i=0,\ldots,r+\ell-2$. Then we have that
\begin{equation}\label{eq:est-e_i-eps}
  \forall\,i=0,\ldots,r+\ell-2\ \forall\, t\ge 0:\ \|e_i(t)\| \le \psi_i(t) - \eps_i(t).
\end{equation}
\end{Lem}
\begin{proof} Let $i\in\{0,\ldots,r+\ell-2\}$. Seeking a contradiction assume that there exists $t_1> 0$ such that $\|e_i(t_1)\| > \psi_i(t_1) - \eps_i(t_1)$. Since we have $\eps_i(0) = \psi_i(0) - \|e_i(0)\|$ it follows that $t_0 := \max \setdef{t\in [0,t_1)}{ \|e_i(t)\| = \psi_i(t) - \eps_i(t)}$ is well defined. Therefore,
\begin{align*}
\|e_i(t)\| &> \psi_i(t) - \eps_i(t)\stackrel{\eqref{eq:est-eps_i}}{\ge} \eps_{i,\max} \ge 0,\\
    k_i(t) &= \frac{1}{1-\varphi_i(t)^2\|e_i(t)\|^2} = \frac{\psi_i(t)^2}{\psi_i(t)^2-\|e_i(t)\|^2}\\
    &\ge \frac{\psi_i(t)}{2 \big(\psi_i(t)-\|e_i(t)\|\big)} \ge \frac{\psi_i(t)}{2 \eps_i(t)}
\end{align*}
for all $t\in(t_0,t_1]$. Hence we find that by~\eqref{eq:fun-con}
\begin{align*}
  \tfrac12 \ddt \|e_i(t)\|^2 &= -k_i(t) \|e_i(t)\|^2 + e_{i+1}(t)^\top e_i(t)\\
  &\le \left( -\frac{\psi_i(t)}{2 \eps_i(t)} \big(\psi_i(t) - \eps_i(t)\big) + \psi_{i+1}(t)\right) \|e_i(t)\|\\
  &\stackrel{\eqref{eq:ODE-eps_i}}{=} \big(\dot \psi_i(t) - \dot \eps_i(t)\big) \|e_i(t)\|
\end{align*}
for almost all $t\in (t_0,t_1]$. Then we have
\begin{align*}
  \|e_i(t_1)\| - \|e_i(t_0)\| &= \int_{t_0}^{t_1} \tfrac12 \|e_i(t)\|^{-1} \ddt \|e_i(t)\|^2 \ds{t} \\
  &\le \int_{t_0}^{t_1} \big(\dot \psi_i(t) - \dot \eps_i(t)\big) \ds{t} \\
  &=  \big(\psi_i(t_1) - \eps_i(t_1)\big) - \big(\psi_i(t_0) - \eps_i(t_0)\big),
\end{align*}
and thus
\[
    0 = \psi_i(t_0) - \eps_i(t_0) - \|e_i(t_0)\| \le  \psi_i(t_1) - \eps_i(t_1) - \|e_i(t_1)\| < 0,
\]
a contradiction.
\end{proof}

\begin{Rem}
We like to emphasize that the estimate~\eqref{eq:est-e_i-eps} also holds for the class of general nonlinear systems of functional differential equations considered in~\cite{BergHoan18} and the proof is the same as given above. This improves the result of~\cite[Thm.~3.1]{BergHoan18}.
\end{Rem}

In the following we utilize the estimate~\eqref{eq:est-e_i-eps} to obtain a bound for the original tracking error~$e(t)$. To this end, we need to introduce some additional notation which is motivated by~\cite[Prop.~3.2]{BergHoan18}. Let $\eps_0,\ldots,\eps_{r+\ell-2}$ be the solutions of~\eqref{eq:ODE-eps_i} and fix $t\ge 0$. Set $N_{i,0}(t) := \psi_{i}(t)$ for $i=0,\ldots,r+\ell-1$ and
\[
    K_{i,0}(t):= \frac{N_{i,0}(t)}{\eps_i(t)},\qquad M_{i,0}(t):= N_{i,0}(t) K_{i,0}(t)
\]
for $i=0,\ldots,r+\ell-2$. Define, for $i=0,\ldots,r+\ell-2$ and $j=0,\ldots,r+\ell-i-1$,
\begin{align*}
  &N_{i,j}(t):= N_{i+1,j-1}(t) + M_{i,j-1}(t),\\
  & L_{i,0}(t):= N_{i,0}(t)^2,\\
  &L_{i,j}(t):=2 \sum_{l=0}^{j-1} \tbinom{j-1}{l} N_{i,l}(t) N_{i,j-l}(t),\\
  &\Phi_{i,0}(t):= \varphi_i(t)^2,\\
  &\Phi_{i,j}(t):=2 \sum_{l=0}^{j-1} \tbinom{j-1}{l} |\varphi_i^{(l)}(t)|\cdot |\varphi_i^{(j-l)}(t)|,\\
  &\Sigma_{i,j}(t):= \tfrac12 \Big( \Phi_{i,0}(t) L_{i,j+1}(t) \!+\! \Phi_{i,1}(t) L_{i,j}(t) \!+\! \Phi_{i,j}(t) L_{i,1}(t)\\
  &\ + L_{i,0}(t) \Phi_{i,j+1}(t)\Big) \\
  &\ + \sum_{l_1=1}^{j-1} \tbinom{j}{l_1} \left( \Phi_{i,l_1}(t) \sum_{l_2=0}^{j-l_1} \tbinom{j-l_1}{l_2} N_{i,l_2}(t) N_{i,j-l_1-l_2}(t)\right.\\
  &\ \left.+ L_{i,j-l_1}(t) \sum_{l_2=0}^{l_1} \tbinom{l_1}{l_2} |\varphi_i^{(l_2)}(t)|\cdot |\varphi_i^{(l_1-l_2)}(t)|\right),\\
  &K_{i,j}(t):= K_{i,0}(t)^2 \Sigma_{i,j-1}(t)\\
  &\ + \sum_{l_1=1}^{j-1} \sum_{l_2=0}^{l_1-1} \tbinom{j-1}{l_1} \tbinom{l_1-1}{l_2} \Sigma_{i,j-l_1-1}(t) K_{i,l_2+1}(t) K_{i,l_1-l_2-1}(t),\\
  &M_{i,j}(t) := \sum_{l=0}^j \tbinom{j}{l} K_{i,l}(t) N_{i,j-l}(t).
\end{align*}
Finally, set
\begin{equation}\label{eq:def-hatKi}
\begin{aligned}
    \hat K_{-1}(t) &:= 0,\\
    \hat K_i(t) &:= \sum_{j=0}^i M_{j,i-j}(t)\quad\text{for}\ \ i=0,\ldots,r+\ell-2.
\end{aligned}
\end{equation}
We stress that $\hat K_i$ depends only on $\varphi_0,\ldots, \varphi_{i+1}$ and the initial errors $e_0(0),\ldots, e_i(0)$, and it is always possible to shape the funnel boundaries $\psi_0,\ldots,\psi_{i+1}$ accordingly to achieve that~$\hat K_i$ is as small as desired. To illustrate this, we calculate that $\hat K_0 = \frac{\psi_0(t)^2}{\varepsilon_0(t)}$ and
\begin{align*}
    \hat K_1(t)  &=  M_{0,1}(t) + M_{1,0}(t)\\
    &= \frac{\psi_0(t)}{\varepsilon_0(t)} \left(\psi_1(t) + \frac{\psi_0(t)^2}{\varepsilon_0(t)}\right) \left(1 + \frac{2\psi_0(t)}{\varepsilon_0(t)}\right) \\
     &\quad + \frac{\psi_1(t)^2}{\varepsilon_1(t)} + \frac{2\psi_0(t)^4 |\dot \varphi_0(t)|}{\eps_0(t)^2}.
\end{align*}
Now, we see that~$\hat K_0(t)$ and~$\hat K_1(t)$ are small provided that~$\psi_0(t), \psi_1(t), |\dot \varphi_0(t)|$ and the ratios $\frac{\psi_0(t)}{\varepsilon_0(t)}$, $\frac{\psi_1(t)}{\varepsilon_1(t)}$ are small. Note that~$\varepsilon_i$ depends on~$\psi_i$,~$\psi_{i+1}$ and~$\|e_i(0)\|$ by~\eqref{eq:ODE-eps_i}. Investigating the behavior for large times~$t$, we find that for typical funnel functions (or, least, they may be designed in this way) we have that $\dot \psi_i(t)\approx 0$, $\psi_i(t)\approx \la_i$ and $\psi_{i+1}(t)\approx \la_{i+1}$, where we use the notation from Lemma~\ref{Lem:est-eps_i}. As a consequence, from~\eqref{eq:ODE-eps_i} we obtain that
\[
    \dot \eps_i(t) \approx -\la_{i+1} + \frac{\la_i (\la_i - \eps_i(t))}{2\eps_i(t)} = - \frac{(2\la_{i+1} + \la_i) \eps_i(t) - \la_i^2}{2\eps_i(t)},
\]
which is solved by $\eps_i(t) = \frac{\la_i^2}{2\la_{i+1}+\la_i}$, thus~$\eps_i$ is approximately equal to this constant for large~$t$. Therefore, we have
\[
    \frac{\psi_i(t)}{\eps_i(t)} \approx \frac{\la_i (2\la_{i+1} + \la_i)}{\la_i^2} = \frac{2\la_{i+1}}{\la_i} + 1.
\]
As a consequence, if $\la_{i+1}\le \beta \la_i$ for some $\beta>0$ and all $i=0,\ldots,r+\ell-2$, then we may guarantee that the ratios $\frac{\psi_i(t)}{\varepsilon_i(t)}$ stay below some a priori known constant for large~$t$ and hence, in order to make~$\hat K_i(t)$ smaller it is indeed sufficient to make $\la_0,\ldots,\la_{i+1}$ smaller in a way that $\la_{j+1}\le \beta \la_j$ is still guaranteed for $j=0,\ldots,i$.

In the example above, choosing $\psi_0(t), \psi_1(t)$ small with $\la_1\le \beta \la_0$ and $\la_2\le \beta\la_1$ while keeping $|\dot \varphi_0(t)|$ small (for~$t$ large enough) establishes any desired quantity for~$\hat K_0(t)$ and~$\hat K_1(t)$.

\begin{Thm}\label{Thm:fun-est-e}
Use the notation and assumptions from Theorem~\ref{Thm:fun-con}.  Let $(x,\eta_{2,\rm ref}):\R_{\ge 0}\to\R^{n+\ell m}$
be the solution of~\eqref{eq:ABC},~\eqref{eq:fun-con}. Then the original tracking error $e(t) = y(t) - y_{\rm ref}(t)$ satisfies
\begin{equation}\label{eq:error-e}
    \forall\, t\ge 0:\ \|e(t)\|\le \sum_{i=1}^{\ell+1} \alpha_i \big(\psi_{i-1}(t) + \hat K_{i-2}(t)\big),
\end{equation}
where $\alpha_i := \|\Gamma K \tilde Q^\ell F_i\|$ for~$F_i$ as in~\eqref{eq:y-ynew}, $i=1,\ldots,\ell$, $\alpha_{\ell+1}:=\|\Gamma\|$ and~$\hat K_i$ is as in~\eqref{eq:def-hatKi}, $i=-1,\ldots,\ell-1$.
\end{Thm}
\begin{proof} By~\eqref{eq:y-ynew}, and a similar equation for~$y_{\rm ref}$ in terms of~$\hat y_{\rm ref}$, both with $\delta=0$ by~(A3), we find that
\[
    e(t) = \Gamma e_0^{(\ell)}(t) + \sum_{i=1}^\ell \Gamma K \tilde Q^\ell F_i e_0^{(i-1)}(t)
\]
for all $t\ge 0$. Inequality~\eqref{eq:error-e} now follows from~\cite[Prop.~3.2]{BergHoan18}, where we use a straightforward extension of this result here: Instead of using the estimate~\eqref{eq:error-evlt} with constant~$\eps_i$ we use~\eqref{eq:est-e_i-eps} with the solutions~$\eps_i(\cdot)$ of~\eqref{eq:ODE-eps_i} and instead of taking the supremum norm of~$\psi_i$ and~$\varphi_i^{(j)}$ in the definition of~$\hat K_i$ we use the values at each~$t\ge 0$. The modification of the proof of~\cite[Prop.~3.2]{BergHoan18} is obvious and omitted.
\end{proof}

We like to highlight that it is a consequence of inequality~\eqref{eq:error-e} that indeed the controller~\eqref{eq:fun-con} achieves prescribed performance of the tracking error $e(t) = y(t) - y_{\rm ref}(t)$ for~\eqref{eq:ABC}, i.e., $\|e(t)\| < \varphi(t)^{-1}$ for all $t\ge 0$. Given any~$\varphi\in\Phi_0$ such that $\varphi(0)\|e(0)\|<1$, we may always choose $\varphi_0,\ldots, \varphi_{r+\ell-1}$ satisfying~\eqref{eq:con-ass} such that
\begin{equation}\label{eq:est-phi_i-phi}
    \forall\, t\ge 0:\quad \sum_{i=1}^{\ell+1} \alpha_i \big(\varphi_{i-1}(t)^{-1} + \hat K_{i-2}(t)\big) < \varphi(t)^{-1},
\end{equation}
since, as illustrated above, it can be achieved that~$\hat K_i$ is as small as desired. For instance, if~$r=\ell=2$ and $\alpha_1,\alpha_2,\alpha_3$ as in Theorem~\ref{Thm:fun-con} are given and, for simplicity, we assume that $e_0(0) = e_1(0) = 0$, then we may choose constant $\varphi_j = \psi_j^{-1} = \la_j^{-1}$, $j=0,1,2$, such that $\la_{1}\le\beta \la_0$ for some $\beta>0$. In this case, the initial value problem~\eqref{eq:ODE-eps_i} is given by
\[
    \dot \eps_j(t) = - \frac{(2\la_{j+1} + \la_j) \eps_j(t) - \la_j^2}{2\eps_j(t)},\quad \eps_j(0)=\la_j
\]
for $j=0,1$. Since $\dot\eps_j(0) = -\la_{j+1} < 0$ and $\hat \eps_j(t) = \frac{\la_j^2}{2\la_{j+1}+\la_j}$ is the equilibrium solution, a simple analysis reveals that~$\eps_j$ is strictly monotonically decreasing with $\lim_{t\to\infty} \eps_j(t) = \frac{\la_j^2}{2\la_{j+1}+\la_j}$. Therefore, we obtain
\begin{align*}
    \hat K_0(t) &\le 2\la_1 + \la_0 \le (2\beta + 1) \la_0,\\
    \hat K_1(t) &\le \frac{2(2\la_1 + \la_0)(\la_1+\la_0)(4\la_1+3\la_0)}{\la_0^2} + \la_1 + 2\la_2\\
    &\le 2(2\beta + 1)(\beta+1)(4\beta+3)\la_0 + \la_1 + 2\la_2.
\end{align*}
Then, we find that $\|e(t)\| < \varphi(t)^{-1}$ holds for all $t\ge 0$, if $\la_0, \la_1, \la_2$ are chosen small enough so that the inequality
\begin{align*}
    & \alpha_1 \la_0 + \alpha_2\big(\la_1 + (2\beta + 1) \la_0\big)\\
    & + \alpha_3 \big(3\la_2 \!+\! \la_1 \!+\! 2(2\beta \!+\! 1)(\beta \!+\! 1)(4\beta\!+\! 3)\la_0 \big) < \inf_{t\ge 0} \varphi(t)^{-1}
\end{align*}
is satisfied.

\begin{Rem}
  We stress that a general construction formula for~$\varphi_0,\ldots,\varphi_{r+\ell-1}$ such that~\eqref{eq:est-phi_i-phi} holds for a given~$\varphi\in\Phi_0$ is not available yet. Further research is necessary to find a suitable way for handling~$\hat K_i$ in~\eqref{eq:est-phi_i-phi}, which depends on~$\varphi_0,\ldots,\varphi_{i+1}$ and~$e_0(0),\ldots,e_i(0)$. Nevertheless, the design parameters may be appropriately adjusted with the help of offline simulations.
\end{Rem}

%
\section{Simulations}\label{Sec:Sim}
%

We illustrate the funnel controller~\eqref{eq:fun-con} by means of a modified linear version of the example discussed in~\cite{DevaChen96}. To this end, consider a system~\eqref{eq:ABC} with
\begin{align*}
    A &= \begin{smallbmatrix} -1 & 1 &0 &0\\ 0 &-3& 0& 1\\ 1& 0& -2& 0\\ 0&0& 3& -1\end{smallbmatrix},\ \ B=\begin{smallbmatrix} 0\\ 2\\ 0\\0\end{smallbmatrix},\ \ C = [1, 0, -3, 0],
\end{align*}
and $d=(0,d_2,0,d_4)^\top\in\cL^{\infty}(\R_{\ge 0}\to\R^4)$, which has strict relative degree $r=2$. The initial value is chosen as $x^0 = 0$ and the reference trajectory as
\[
    y_{\rm ref}(t) = \left\{ \begin{array}{rl} (1-\cos t), & t\in[0,2\pi],\\ 0, & t>2\pi.\end{array}\right.
\]
Clearly, we have $y_{\rm ref}\in\cW^{1,\infty}(\R_{\ge 0}\to\R)$. As disturbances we consider
\[
    d_2(t) = \tfrac12\sin(5t) + \cos(8t),\quad d_4(t)=\sin(6t)+\tfrac12\cos(4t)
\]
for $t\ge 0$. In order to determine the new output as in~\eqref{eq:new-output} we need to transform the system into Byrnes-Isidori form~\eqref{eq:BIF}. With $\eta_1 = x_4$ and $\eta_2 = x_3$ this form is given by
\begin{align*}
  \ddot y(t) &= -18y(t)-7\dot y(t) + \eta_1(t) - 24\eta_2(t) + 2u(t)+d_2(t),\\
  \dot \eta_1(t) &= -\eta_1(t) + 3\eta_2(t) +d_4(t),\\
  \dot \eta_2(t) &= \eta_2(t) + y(t).
\end{align*}
It is now easy to see that assumption~(A1) is satisfied with the choice $\ell=1$ and $\tilde Q = \tilde P = 1$, assumption~(A2) is satisfied since $k_3=0$, and assumption~(A3) is satisfied as $\delta=0$. Hence,~$K$ as in~\eqref{eq:defK} is given by $K = \Gamma^{-1} \tilde P^{-1} = \tfrac{1}{2}$ and the new output~\eqref{eq:new-output} is $y_{\rm new}(t) = \tfrac12 \eta_2(t) = \tfrac12 x_3(t)$. The initial value $\eta_{2,\rm ref}^0$ as in~\eqref{eq:eta-ref-0} needed for the controller~\eqref{eq:fun-con} can be computed as
\begin{align*}
    \eta_{2,\rm ref}^0 = - \int_0^{2\pi} e^{-s} (1 - \cos s) \ds{s} =  -\frac{1069}{238}.
\end{align*}
The funnel functions are chosen as
\begin{align*}
  \varphi_0(t) &= \big(e^{-2t} + 0.01\big)^{-1},\qquad \varphi_1(t) = \big(2 e^{-2t} + 0.01\big)^{-1},\\
  \varphi_2(t) &= \big(2 e^{-10t} + 0.01\big)^{-1},
\end{align*}
and clearly~\eqref{eq:con-ass} is satisfied and the initial errors $e_0(0), e_1(0), e_2(0)$ lie within the respective funnel boundaries. Therefore, feasibility of the controller~\eqref{eq:fun-con} is guaranteed by Theorem~\ref{Thm:fun-con}.

The bound~\eqref{eq:error-e} for the original tracking error~$e = y - y_{\rm ref}$ as given in Theorem~\ref{Thm:fun-est-e} reads as follows:
\[
    \|e(t)\| \le \alpha_1 \psi_0(t) + \alpha_2 \big(\psi_1(t) + \hat K_0(t)\big) =: \Psi(t),\quad t\ge 0,
\]
where $\psi_i(\cdot) = \varphi_i(\cdot)^{-1}$ for $i=0,1$. Clearly, $\alpha_2 = \|\Gamma\| = 2$ and we  calculate that $\alpha_1 = \|\Gamma K \tilde Q F_1\| = 2$. Then we obtain that
\[
    \|e(t)\|\le \Psi(t) = 2\left( \psi_0(t) + \psi_1(t) + \frac{\psi_0(t)^2}{\eps_0(t)}\right),
\]
where $\eps_0$ is the solution of~\eqref{eq:ODE-eps_i} for $i=0$.

\captionsetup[subfloat]{labelformat=empty}
\begin{figure}[h!tb]
  \centering
  \subfloat[Fig.~\ref{fig:sim}a: Output and reference trajectory]
{
\centering
  \hspace*{-5mm} \includegraphics[trim=20 10 20 10,clip,width=7.5cm]{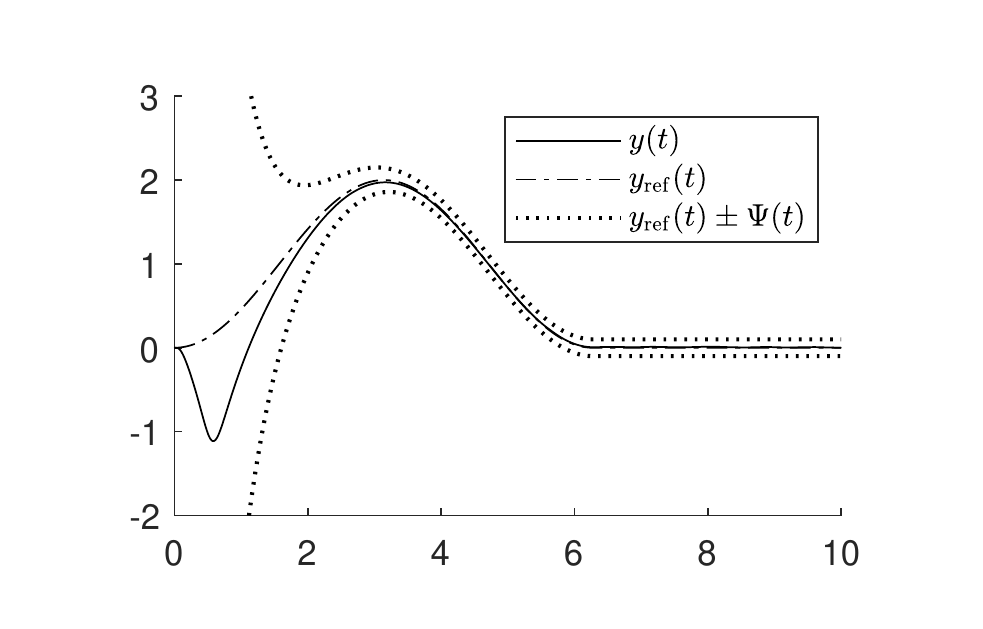}
\label{fig:sim-y}
}\\[-1mm]
\subfloat[Fig.~\ref{fig:sim}b: States]
{
\centering
 \hspace*{-5mm} \includegraphics[trim=20 10 20 10,clip,width=7.5cm]{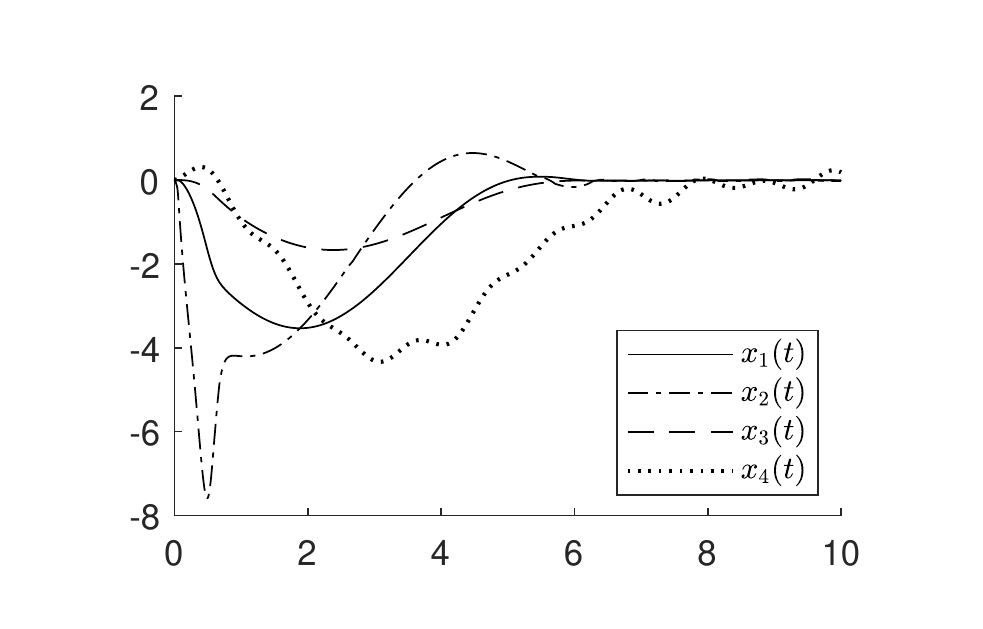}
\label{fig:sim-x}
}\\[-1mm]
\subfloat[Fig.~\ref{fig:sim}c: Input function]
{
\centering
 \hspace*{-5mm} \includegraphics[trim=20 10 20 10,clip,width=7.5cm]{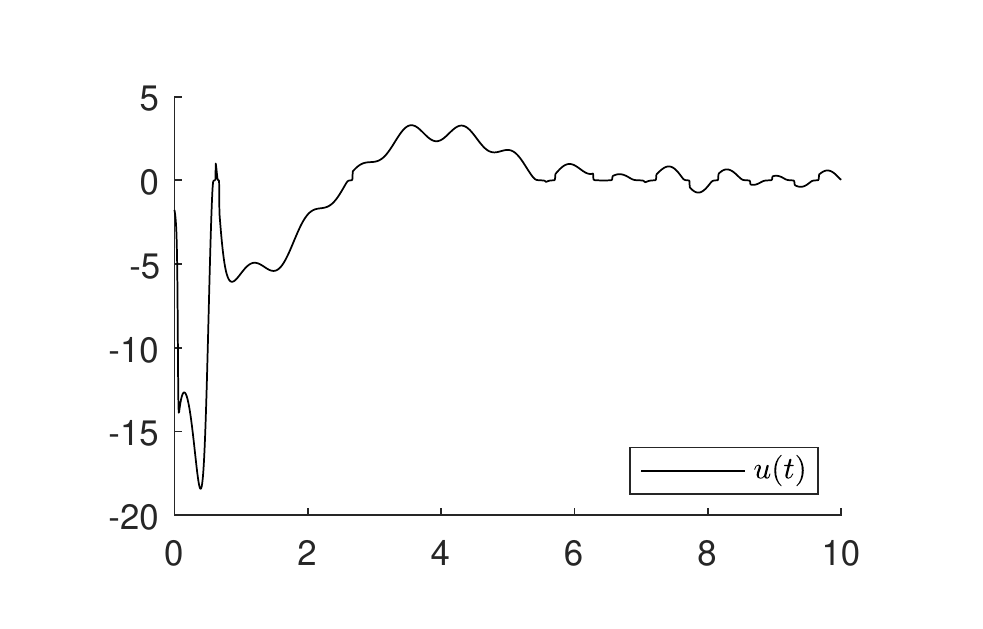}
\label{fig:sim-u}
}
\caption{Simulation of the controller~\eqref{eq:fun-con} for system~\eqref{eq:ABC}.}
\vspace*{-5mm}
\label{fig:sim}
\end{figure}

The simulation of the controller~\eqref{eq:fun-con} applied to system~\eqref{eq:ABC} over the time interval $[0,10]$ has been performed in MATLAB (solver: {\tt ode15s}, rel.\ tol.: $10^{-8}$, abs.\ tol.: $10^{-7}$) and is depicted in Fig.~\ref{fig:sim}. Fig.~\ref{fig:sim-y} shows the original output~$y$, the reference signal~$y_{\rm ref}$ and the bounds~$y_{\rm ref}\pm\Psi$ for the output. The states are depicted in Fig.~\ref{fig:sim-x} and the input in Fig.~\ref{fig:sim-u}. It can be seen that, even in the presence of the disturbances, a prescribed performance of the tracking error can be achieved with the funnel controller~\eqref{eq:fun-con}, while at the same time the generated input is bounded and shows an acceptable performance as well.

%
\section{Conclusion}\label{Sec:Concl}
%

In the present paper we proposed a novel controller for achieving tracking with prescribed performance of the tracking error for uncertain linear non-minimum phase systems. Our approach is based on the construction of a new output for the system given by~\eqref{eq:new-output} to which the recently developed funnel controller from~\cite{BergHoan18} is applied. To guarantee feasibility a new reference signal needs to be calculated as well, which is given by the solution of~\eqref{eq:new-ref} with initial value~\eqref{eq:eta-ref-0}. Approximating the reference signal (on an interval of interest) by an exosystem of the form~\eqref{eq:exo}, we may compute the latter initial value via the solution of a Sylvester equation. The resulting controller~\eqref{eq:fun-con} is shown to be feasible in Theorem~\ref{Thm:fun-con}, independent of the disturbance~$d$. Bounds for the original tracking error have been derived in Theorem~\ref{Thm:fun-est-e}. It has been shown that, by appropriately designing the funnel functions, these bounds can be adjusted to be as small as desired. At the same time, the input~$u$ generated by~\eqref{eq:fun-con} remains bounded.

We stress that some features of funnel control (see e.g.~\cite{BergHoan18,BergReis18b,IlchRyan02b}) are lost with this approach: The controller~\eqref{eq:fun-con} is not model-free in general, since knowledge of~$\tilde Q$ and~$\tilde P$ is required to determine the new output~\eqref{eq:new-output} and the new reference~\eqref{eq:new-ref}. Furthermore, measurement of~$y, \dot y,\ldots, y^{(r-1)}$ is not sufficient, but it is required that additional state variables can be measured so that $y_{\rm new}, \dot y_{\rm new}, \ldots, y_{\rm new}^{(r+\ell-1)}$ are available to the controller. However, knowledge of the full initial value~$x^0$ and the disturbance~$d(\cdot)$ are not required, cf.\ Section~\ref{Ssec:ContrObj}.

While the controller design~\eqref{eq:fun-con} is robust with respect to disturbances satisfying~\eqref{eq:rel-deg} and~(A3), further research is necessary to extend this class.

\section*{Acknowledgement}

I am indebted to Achim Ilchmann (Technische Universit\"at Ilmenau) and Timo Reis (Universtit\"at Hamburg) for several constructive discussions.

\bibliographystyle{elsarticle-harv}

%
%
%

\end{document}